\documentclass{amsart}%
\usepackage{amsmath}
\usepackage{amsfonts}
\usepackage{amssymb}
\usepackage{graphicx}
\usepackage{latexsym}
\usepackage{amscd}
\usepackage{tikz-cd}
\usepackage{enumitem}
\usepackage[all, cmtip]{xy}
\usepackage[utf8]{inputenc}
\usepackage{amsthm}
\usepackage{xcolor}
\usepackage[a4paper,  margin=1in]{geometry}
\usepackage{epic}
\usepackage{pstricks}
\usepackage[all]{xy}
\usepackage{graphicx}
\usepackage{epic,eepic}
\usepackage{epsfig}
\usepackage{float}
\usepackage{extarrows}
\usepackage{tikz}
\usepackage{calc}%
\usepackage{comment}
\providecommand{\U}[1]{\protect\rule{.1in}{.1in}}
\providecommand{\U}[1]{\protect\rule{.1in}{.1in}}
\providecommand{\U}[1]{\protect\rule{.1in}{.1in}}
\newtheorem{theorem}{Theorem}[section]

\newtheorem{conjecture}[theorem]{Conjecture}
\newtheorem{corollary}[theorem]{Corollary}

\newtheorem{definition}[theorem]{Definition}
\newtheorem{example}[theorem]{Example}

\newtheorem{lemma}[theorem]{Lemma}

\newtheorem{question}[theorem]{Question}
\newtheorem{proposition}[theorem]{Proposition}
\newtheorem{remark}[theorem]{Remark}

\DeclareMathOperator{\Aut}{Aut}
\DeclareMathOperator{\Mod}{Mod-}
\DeclareMathOperator{\Gr}{Gr-}
\DeclareMathOperator{\QGr}{QGr-}
\DeclareMathOperator{\Fdim}{Fdim-}
\DeclareMathOperator{\GKdim}{GKdim}

\def\L#1{\displaystyle\lim_{\rightarrow}#1}

\def\alg{\mathop{\text{\bf Alg}}}
\def\pl{\mathop{\hbox{\rm P}_l}}

\def\soc{\mathop{\hbox{\rm Soc}}}

\def\a{\alpha}

\def\b{\beta}

\def\remove#1{}

\def\path{\mathop{\hbox{\rm Path}}}

\newcommand{\N}{{\mathbb{N}}}
\newcommand{\Z}{{\mathbb{Z}}}
\newcommand{\K}{{\mathbb{K}}}

\DeclareMathOperator{\Span}{span}
\usetikzlibrary{decorations.markings}

\newcommand{\Sink}{\operatorname{Sink}}

\tikzstyle{vertex}=[circle, draw, fill, inner sep=0pt, minimum size=6pt]
\usepackage[symbol]{footmisc}
\begin{document}
\title[Leavitt Path Algebra over Kronecker Square of Quivers]{Leavitt Path Algebra over Kronecker Square of Quivers and Cross product algebra}
\author{Jehan Alarfaj}
\address{Department of Mathematics and Statistics, Saint Louis University, St. Louis,
MO-63103, USA, Department of Mathematics, College of Science, Imam Abdulrahman Bin Faisal University, P.O. Box 1982, Dammam 31441, Saudi Arabia}
\email{jehan.alarfaj@slu.edu}
\author{Dolores Mart\'{i}n Barquero}
\address{Departamento de Matem\'{a}tica Aplicada, Escuela de Ingenier\'\i as Industriales, Campus de Teatinos, Universidad de M\'{a}laga, 29071 Málaga.  Spain}
\email{dmartin@uma.es}
\author{Ashish K. Srivastava}
\thanks{Ashish K Srivastava is the corresponding author}
\address{Department of Mathematics and Statistics, Saint Louis University, St. Louis,
MO-63103, USA}
\email{ashish.srivastava@slu.edu}
\thanks{The second author is supported by the Spanish Ministerio de Ciencia, Innovaci\'on  y Universidades through project  PID2023-152673NB-I00 and by the Junta de Andaluc\'{\i}a  through project PPRO-FQM336-G-2023 (FQM336-G-FEDER), all of them with FEDER funds. }
\begin{abstract}
In this paper, we initiate the study of Leavitt path algebra over Kronecker square of a quiver and show the similarities and contrasts in the properties of Leavitt path algebra over a quiver and its Kronecker square. Furthermore, we discuss the connection of Leavitt path algebra over Kronecker square of a quiver with Hayashi's face algebra and the cross product algebra construction of the Leavitt path algebra over the original quiver.
\end{abstract}
\maketitle

\section{Introduction}

\noindent This paper initiates the study of Leavitt path algebra over Kronecker square of quivers which is inspired by the work of Corrales Garc\'{i}a et al. in \cite{GBG} and Calder\'{o}n and Walton in \cite{CW}. Corrales Garc\'{i}a et al. \cite{GBG} give construction of cross product algebra $L_{\mathbb K}(Q_1) \otimes_G L_{\mathbb K}(Q_2)$ for any two quivers $Q_1$, $Q_2$ (here $G$ is the group of units of $\mathbb K[x, x^{-1}]$) and show that if $Q_1$ and $Q_2$ are row-finite quivers with no sinks, then there is an isomorphism between the cross product algebra $L_{\mathbb K}(Q_1) \otimes_G L_{\mathbb K}(Q_2)$ and the Leavitt path algebra $L_{\mathbb K}(Q_1\times Q_2)$ where $Q_1\times Q_2$ is not the usual product of quivers but a special kind of product. In fact, this special product $Q_1\times Q_2$ mentioned in \cite{GBG} turns out to be the Kronecker product of quivers $Q_1$ and $Q_2$. In particular, this shows that if $Q$ is a row-finite quiver with no sink, then the cross product algebra $L_{\mathbb K}(Q) \otimes_G L_{\mathbb K}(Q)$ is isomorphic to $L_{\mathbb K}(\widehat{Q})$, where $\widehat{Q}$ is the Kronecker square of $Q$. In \cite[Proposition 3.4]{CW}, the authors realize the face algebra $\mathcal{H}_{\mathbb K}(Q)$ of a finite quiver $Q$ defined by Hayashi in \cite{H} as path algebra $\mathbb K\widehat{Q}$ over Kronecker square $\widehat{Q}$ of $Q$.  
  
 A quiver $Q=(Q_{0}, Q_{1}, s, r)$, which is mostly called directed graph in Leavitt path algebra literature, consists of two sets $Q_{0}$ and $Q_{1}$ together with maps
$s,r:Q_{1}\rightarrow Q_{0}$, called source and range, respectively. The elements of $Q_{0}$ are called
\textit{vertices} and the elements of $Q_{1}$ are called \textit{arrows}. The \textit{path algebra} of $Q$ over a field $\mathbb K$ is the $\mathbb K$-algebra with $\mathbb K$-basis given by all the paths in $Q$ and ring structure determined by path concatenation. We will read paths of $Q$ from left-to-right. For any vertex $i\in Q_0$, we associate a trivial path $e_i$. The path algebra $\mathbb K Q$ is $\mathbb N$-graded by path length, where $(\mathbb KQ)_k=\mathbb K(Q)_k$, with $Q_k$ consisting of paths of length $k\in \mathbb N$.      

Hayashi's face algebra $\mathcal{H}_{\mathbb K}(Q)$ for a finite quiver $Q$ over a field $\mathbb K$ is generated by elements $\{x_{a, b}\}$, for $a, b \in Q_k$ subject to the following relations:

(1) $x_{i, j} x_{u, v}=\delta_{i, u} \delta_{j, v} x_{i, j}$, for each $i, j, u, v\in Q_0$.

(2) $x_{s(p), s(q)} x_{p, q} =x_{p, q} =x_{p, q} x_{r(p), r(q)}$, for each $p, q \in Q_1$. 

(3) $x_{p, q} x_{p', q'} =\delta_{r(p), s(p')} \delta_{r(q), s(q')} x_{pp', qq'}$, for each $p, q, p', q' \in Q_1$.

 $\mathcal{H}_{\mathbb K}(Q)$ is a unital $\mathbb K$-algebra, with unit given by $1=\Sigma_{i, j\in Q_0} x_{i, j}$. Let us denote by $Q_i$, the set of paths of length $i\in \mathbb N$. Clearly $\mathcal{H}_{\mathbb K}(Q)$ has an $\mathbb N$-grading given by $\mathcal{H}_{\mathbb K}(Q)=\oplus_{i\in \mathbb N} (\mathcal{H}_{\mathbb K}(Q))_i$.

Hayashi introduced face algebras as a new class of quantum groups inspired by the quantum inverse scattering method and solvable lattice models of face type. Face algebras also arise in Jones' index theory of subfactors. Face algebras are related to Fusion rules of Wess-Zumino-Witten model in Conformal Field Theory. The class of Face algebras contains all bialgebras and just like bialgebras, face algebras also produce monoidal categories as their (co-)module categories. 

Recently Huang, Walton, Wicks, and Won have shown that the weak bialgebras that coact universally on the path algebra $\mathbb KQ$ (either from the left, from the right, or from both directions compatibly) are each isomorphic to Hayashi's face algebra $\mathcal{H}_{\mathbb K}(Q)$. In \cite[Conjecture 1.6]{HWWW} it is conjectured that if $H$ is a finite-dimensional weak bialgebra over an algebraically closed field $\mathbb K$ with commutative counital subalgebras, then $H$ is isomorphic to a weak bialgebra quotient of face algebra $\mathcal{H}_{\mathbb K}(Q)$ for some finite quiver $Q$. This motivates us to study quotients of face algebra. We define an algebra $\mathcal{A}_{\mathbb K}(Q)$ in Section 5, which is a quotient algebra of face algebra $\mathcal{H}_{\mathbb K}(\overline{Q})$ over a quiver $\overline{Q}$, which is obtained by adding for each arrow in $Q$, another arrow in the opposite direction as well. This algebra $\mathcal{A}_{\mathbb K}(Q)$ turns out to be isomorphic to Leavitt path algebra over a new quiver called the Kronecker square $\widehat{Q}$. As we already mentioned that if $Q$ is a row-finite quiver with no sink, then the cross product algebra $L_{\mathbb K}(Q) \otimes_G L_{\mathbb K}(Q)$ studied in \cite{GBG} is isomorphic to $L_{\mathbb K}(\widehat{Q})$. All this motivates us to study the Leavitt path algebra over Kronecker square of quivers.   

Leavitt path algebras are algebraic analogues of graph C*-algebras and are also natural generalizations of Leavitt algebra of type $($$1,n$$)$ constructed by William Leavitt. 

Given an arbitrary quiver $Q=(Q_{0}, Q_{1}, s, r)$, for each arrow $e\in Q_{1}$, we consider an arrow $e^{\ast}$ in the opposite direction, called the ghost arrow. So, we have $r(e^{\ast})=s(e)$, and $s(e^{\ast})=r(e)$. Let $\mathbb K$ be a field. The \textit{Leavitt path
algebra }$L_{\mathbb K}(Q)$ is defined as the $\mathbb K$-algebra generated by $\{v, e, e^{\ast}: v\in Q_0, e\in Q_1\}$ subject to the following relations:

(1) $vw=\delta_{v, w}v$, for all $v, w\in Q_0$.

(2) $s(e)e=e=er(e)$ for all $e\in Q_{1}$.

(3) $r(e)e^{\ast}=e^{\ast}=e^{\ast}s(e)$\ for all $e\in Q_{1}$.

(4) (The ``CK-1 relations") For all $e,f\in Q_{1}$, $e^{\ast}f=\delta_{e, f} r(e)$.

(5) (The ``CK-2 relations") For every regular vertex (that is, vertex which emits at least one and only finitely many arrows) $v\in Q_{0}$, 
\begin{equation*}
v=\sum_{e\in Q_{1},s(e)=v}ee^{\ast}. 
\end{equation*}

\noindent The Leavitt path algebra $L_{\mathbb K}(Q)$ has the following universal property: if $A$ is any $\mathbb K$-algebra generated by a family of elements $\{a_i, b_p, c_{p^{\ast}}: i\in Q_0, p\in Q_1\}$ satisfying the relations analogous to (1) - (5) above, then there exists a unique $\mathbb K$-algebra homomorphism $\varphi:L_{\mathbb K}(Q)\longrightarrow A$ given by $\varphi(i)=a_i$, $\varphi(p)=b_p$, and $\varphi(p^{\ast})=c_{p^{\ast}}$. 

It can be shown that $L_{\mathbb K}(Q)$ is spanned as a $\mathbb K$-vector space by $\{xy^{\ast}: r(x)=r(y)\}$ where $x$ is a path and $y^*$ is a ghost path. 
Every Leavitt path algebra $L_{\mathbb K}(Q)$ is a $\mathbb{Z} $\textit{-graded
algebra}, namely, $L_{\mathbb K}(Q)={\displaystyle\bigoplus\limits_{n\in\mathbb{Z}}}
L_{n}$ induced by defining, for all $v\in Q_{0}$ and $e\in Q_{1}$, $\deg
(v)=0$, $\deg(e)=1$, $\deg(e^{\ast})=-1$. For each $n\in\mathbb{Z}$, the \textit{homogeneous component }$L_{n}$
is given by
\[
L_{n}=\{ {\textstyle\sum} k_{i}\alpha_{i}\beta_{i}^{\ast}:\text{ }%
|\alpha_{i}|-|\beta_{i}|=n\}.
\]

\noindent For more details on Leavitt path
algebras, we refer the reader to \cite{AAS}. 

\section{Kronecker square of a quiver and its combinatorics}

\noindent  The Kronecker product of quivers was defined by Weichsel in \cite{W}. He describes the Kronecker product of two finite quivers $Q_1$ and $Q_2$ as a quiver $Q_1 \otimes Q_2$ whose adjacency matrix $A_{Q_1\otimes Q_2}$ is the Kronecker product of adjacency matrices $A_{Q_1} $and $A_{Q_2}$. Recall that if $A_{Q_1}=[a_{ij}]_m $ and $A_{Q_2}=[b_{ij}]_n $, then the Kronecker product of these matrices is given by $A_{Q_1\otimes Q_2}= [c_{ij}]_{mn}$ where $c_{ij}= a_{ij} [b_{ij}]$.  The Kronecker product of a quiver $Q$ with itself, $Q \otimes Q$, is called the Kronecker square of $Q$ and denoted as $\widehat Q$.

\noindent The Kronecker product of quivers is extremely useful in network modeling as has been shown in \cite{ML} that using Kronecker product we can generate quivers which can effectively model the structure of real networks. In this section, we list some simple observations about the combinatorial relationship between a quiver and its Kronecker square. Most of these observations can be easily generalized to the more general setting of Kronecker product of two distinct quivers as well.     

\noindent Let $Q=(Q_{0}, Q_{1}, s, r)$ be a quiver. The Kronecker square $\widehat{Q}=(\widehat{Q_{0}}, \widehat{Q_{1}}, \widehat{s}, \widehat{r})$ is given by 

\[\widehat{Q_0}=\{[v_i, v_j]: v_i, v_j \in Q_0\}\]
  
\[\widehat{Q_1}=\{[e_i, e_j]: e_i, e_j \in Q_1\}\]

\[\widehat{s}([e, f])=[s(e), s(f)]\]

\[\widehat{r}([e, f])=[r(e), r(f)]\]

\noindent Any path in $\widehat{Q}$ is of the form $[e_1e_2\cdots e_k, f_1f_2\cdots f_k]:=[e_1, f_1][e_2, f_2] \cdots [e_k, f_k]$. A quiver $Q$ and its Kronecker square $\widehat{Q}$ share many interesting graph-theoretic properties which we list below.

\begin{enumerate}

\item $Q$ is finite if and only if $\widehat{Q}$ is finite. In fact, $|\widehat{Q_0}|=|Q_0|^2$ and $|\widehat{Q_1}|=|Q_1|^2$.

\medskip

\item A path $\mu$ $=e_{1}\dots e_{n}$ in $Q$ is called \textit{closed} if $r(e_{n})=s(e_{1})$. The closed path $\mu$ is called a
\textit{cycle} if it does not pass through any of its vertices twice, that is,
if $s(e_{i})\neq s(e_{j})$ for every $i\neq j$. 

It may be noted that $Q$ is acyclic if and only if $\widehat{Q}$ is acyclic \cite{CW}. This follows from the observation that $Q$ may be identified as a subquiver of $\widehat{Q}$ and if  $[e_1e_2\ldots e_k, f_1 f_2\ldots f_k]$ is a cycle in $\widehat{Q}$, then $e_1e_2\ldots e_k$ and $f_1f_2\ldots f_k$ are both cycles in $Q$ because $[s(e_1), s(f_1)]=\hat{s}([e_1, f_1])=\hat{r}([e_k, f_k])=[r(e_k), r(f_k)]$.

\medskip

\item An \textit{exit }for a path $\mu=e_{1}\dots e_{n}$ is an
arrow $e$ such that $s(e)=s(e_{i})$ for some $i$ and $e\neq e_{i}$. A quiver $Q$
is said to satisfy \textit{Condition $($L$)$} if every cycle in $Q$ has an exit. 

A quiver $Q$ satisfies the Condition (L) if and only if its Kronecker square $\widehat{Q}$ satisfies the Condition (L).

\begin{proof}
Suppose $Q$ satisfies the Condition (L). Let $\widehat{C}=[e_1e_2\ldots e_k, f_1 f_2\ldots f_k]$ be a cycle in $\widehat{Q}$. Then as observed above $e_1e_2\ldots e_k$ and $f_1f_2\ldots f_k$ are closed paths in $Q$. Note that a closed path in $Q$ need not be a cycle, but every closed path traverses at least one cycle. In particular, $e_1e_2\ldots e_k$ contains a cycle $C$ in $Q$. Thus there is an index $i$ such that the edge $e_i$ lies on $C$, and $s(e_i)$ is a vertex of $C$. Since $Q$ satisfies the Condition (L), the cycle $C$ has an exit in $Q$. So there exists an edge $e'$ in $Q$ with $s(e') = s(e_i)$ and $e' \neq e_i$. Now consider the edge $[e', f_i]$ in $\widehat{Q}$. We have
\begin{align*}
    s([e', f_i]) &= [s(e'), s(f_i)] = [s(e_i), s(f_i)]=s([e_i, f_i]), \\
    [e', f_i]  &\neq [e_i, f_i] \quad \text{since } e' \neq e_i.
\end{align*}
Hence $[e', f_i]$ is an exit for $\widehat{C}$. This shows that $\widehat{Q}$ satisfies the Condition (L). 

Conversely, assume that $\widehat{Q}$ satisfies the Condition (L) and let $c=e_1\cdots e_k$ be a cycle in $Q$. Note that $[e_1, e_1][e_2, e_2] \cdots [e_k, e_k]$ is a cycle in $\widehat{Q}$ and as $\widehat{Q}$ satisfies the Condition (L), there exists an arrow $[e', f']$ with $[e', f'] \neq [e_i, f_i]$ and $[s(e'), s(f')]=[s(e_i), s(e_i)]$. This yields an exit $e'$ with $e'\neq e_i$ and $s(e')=s(e_i)$ for $c$ in $Q$. This proves that $Q$ satisfies the Condition (L).   
\end{proof}

\medskip

\item A quiver $Q$ is said to satisfy \textit{Condition $($K$)$}, if any vertex $v$
on a simple closed path $c$ is also the base of another simple closed path $c^{\prime}$ different from $c$. It is not difficult to verify that $Q$ satisfies the Condition (K) if and only if $\widehat{Q}$ satisfies the Condition (K). 

\medskip

\item A vertex $v$ is called a \textit{source} if no arrow ends at it and it is called a \textit{sink} if no arrow begins at it. A quiver $Q$ has no source or sink if and only if its Kronecker square $\widehat{Q}$ has no source or sink.

\begin{proof}
Assume that $Q$ has no source or sink. We claim that $\widehat{Q}$ has no source or sink as well. Note that if $[v, w]$ is a source in $\widehat{Q}$, then there is no arrow from any vertex $[v', w']$ to $[v, w]$ for $v', w'\in Q_0$. This means either there is no arrow from any vertex $v'$ to $v$ or there is no arrow from any vertex $w'$ to $w$. This makes either $v$ a source vertex or $w$ a source vertex. Similarly, if $[v, w]$ is a sink in $\widehat{Q}$, then there is no arrow from $[v, w]$ to any vertex $[v'', w'']$ for $v'', w''\in Q_0$. This means either there is no arrow from vertex $v$ to any vertex $v''$ or there is no arrow from vertex $w$ to any vertex $w''$. This makes either $v$ a sink vertex or $w$ a sink vertex. In both the cases, we have a contradiction. This establishes our claim that $\widehat{Q}$ has no source or sink as well.   

Conversely, assume $\widehat{Q}$ has no source or sink. We claim that then $Q$ has no source or sink. Suppose $v$ is a source in $Q$. Then there is no arrow from any vertex $v'$ to $v$. Clearly, then for any vertex $w$ in $Q_0$, $[v, w]$ is a source in $\widehat{Q}$, a contradiction. Similarly, we can see that if $v$ is a sink in $Q$, then for any vertex $w$ in $Q_0$, $[v, w]$ is a sink in $\widehat{Q}$, a contradiction again. This proves that $Q$ has no source or sink as well.          
\end{proof}

\end{enumerate}

\section{Leavitt path algebra over Kronecker square}

\noindent Let $Q=(Q_{0}, Q_{1}, s, r)$ be a quiver and its Kronecker square be $\widehat{Q}=(\widehat{Q_{0}}, \widehat{Q_{1}}, \widehat{s}, \widehat{r})$. Let $\mathbb K$ be any field. The Leavitt path algebra $L_{\mathbb K}(\widehat{Q})$ is a $\mathbb K$-algebra generated by $\{[v_i, v_j], [e_i, e_j], [e_i, e_j]^{\ast}: v_i\in Q_0, e_i \in Q_1\}$ subject to the following relations:

\begin{enumerate}

\item $[v_i, v_j][v_{i'}, v_{j'}] =\delta_{i, i'} \delta_{j, j'} [v_i, v_j]$ for each $v_i, v_j, v_{i'}, v_{j'}\in Q_0$.

\smallskip

\item $[s(e), s(f)][e, f]=[e, f]$ and $[e, f][r(e), r(f)]=[e, f]$ for each $e, f\in Q_1$.

\smallskip

\item $[e_i, e_j]^{\ast} [e_{i'}, e_{j'}]=\delta_{i, i'} \delta_{j, j'} [r(e_i), r(e_j)]$ for each $e_i, e_j, e_{i'}, e_{j'}\in Q_1$.
 
 \smallskip
 
\item If $v_i, v_j$ are both vertices that emit at least one and only finitely many arrows, then $[v_i, v_j]= \Sigma_{v_i=s(e_i), v_j=s(e_j)}[e_i, e_j] [e_i, e_j]^{\ast}$. 

\end{enumerate}

\begin{example} \rm

Let $Q$ be the Dynkin diagram $A_2$. That is,

\[
\xymatrix{ && v_1\ar@{->}[rr]^e && v_2}
\]

Then it is known that $L_{\mathbb K}(Q)\cong M_2(\mathbb K)$. Now, in this case the Kronecker quiver $\widehat{Q}$ is given as follows;

\[
\xymatrix{[v_1, v_1] \ar@{->}[drr]^{[e, e]} && [v_1, v_2]  \\
[v_2, v_1] && [v_2, v_2]}
\]

It is not difficult to verify that $L_{\mathbb K}(\widehat{Q}) \cong M_2(\mathbb K) \times \mathbb K^2$. 

\end{example}

\begin{example} \rm

For any integer $n\ge 2$, we let $R_n$ denote the rose quiver with one vertex and $n$ loops:
$$R_n = \xymatrix{ & {\bullet^v} \ar@(ur,dr)^{e_1}  \ar@(u,r)^{e_2} \ar@(ul,ur)^{e_3}  \ar@{.} @(l,u) \ar@{.} @(dr,dl)
	\ar@(r,d)^{e_n}  \ar@{}[l] ^{\hdots} } \ \ $$ Then $L_{\mathbb K}(R_n)$ 	is defined to be the $\mathbb K$-algebra generated by $v$, $e_1, \ldots, e_n$, $e^*_1, \ldots, e^*_n$, satisfying the following relations 
\begin{center}
	$v^2 =v, ve_i = e_i= e_iv$, $ve^*_i = e^*_i = e^*_i v$, $e^*_i e_j = \delta_{i,j} v$ and $\sum^n_{i=1}e_ie^*_i =v$	
\end{center}
for all $1\le i, j\le n$.  

The \textit{Leavitt $\mathbb K$-algebra of type} $(1;n)$, denoted by $L_{\mathbb K}(1, n)$,  is the free associative $\mathbb K$-algebra on the $2n$ variables $x_1, \hdots, x_n, y_1, \hdots, y_n$ subject to the relations $ \sum^{n}_{i=1}x_iy_i =1$ and $y_ix_j = \delta_{i,j}1 \, (1\le i, j\le n)$. Clearly, $L_{\mathbb K}(1, n)\cong L_{\mathbb K}(R_n)$ as $\mathbb K$-algebras, by the mapping: $1\longmapsto v$, $x_i\longmapsto e_i$	and $y_i\longmapsto e^*_i$ for all $1\le i\le n$.   

Note that the Kronecker square $\widehat{R_n}$ is another rose quiver with one vertex $[v, v]$ and $n^2$ loops $[e_i, e_j]$, $1\le i, j\le n$. We have $L_{\mathbb K}(\widehat{R_n}) \cong L_{\mathbb K}(R_{n^2}) \cong L_{\mathbb K}(1, n^2)$ under the map $$\varphi: L_{\mathbb K}(\widehat{R_n}) \rightarrow L_{\mathbb K}(1, n^2)$$ defined as

\medskip

 $\varphi([v, v])=1$; $\varphi([e_i, e_j])=x_i$ if $i=j$; $\varphi([e_i, e_j])=x_{(n-i+j)n+i}$ if $i>j$; and $\varphi([e_i, e_j])=x_{(j-i)n+i}$ if $i<j$, and

\medskip

$\varphi([e_i, e_j]^{\ast})=y_i$ if $i=j$; $\varphi([e_i, e_j]^{\ast})=y_{(n-i+j)n+i}$ if $i>j$; and $\varphi([e_i, e_j]^{\ast})=y_{(j-i)n+i}$ if $i<j$.       

\end{example}

\section{Cross product of Leavitt path algebras and Leavitt path algebra over Kronecker square}

\noindent Corrales Garc\'{i}a et al in \cite{GBG} study the cross product of Leavitt path algebras using gauge action. For a field $\mathbb K$, let $R=\mathbb K[x, x^{-1}]$ be the Laurent polynomial $\mathbb K$-algebra and $G=U(R)$ be the group of invertible elements. For any $\mathbb K$-algebra $A$, denote by $A_R$ the extension by scalars. Each $\mathbb Z$-grading on a $\mathbb K$-algebra $A$ comes from a group homomorphism $\rho: G\rightarrow \Aut(A_R)$ such that the homogeneous elements of degree $n$ are precisely $x\in A$ such that $\rho(z)(x\otimes 1)=x\otimes z^n$ for any $z\in G$. Then instead of giving a $\mathbb Z$-grading on $A$, we may consider a group homomorphism $\rho: G\rightarrow \Aut(A_R)$. Thus for two $\mathbb Z$-graded $\mathbb K$-algebras $A$ and $B$ with associated homomorphisms $\rho:G\rightarrow \Aut(A_R)$ and $\sigma: G\rightarrow \Aut(B_R)$, we can construct a homomorphism $\rho \otimes \sigma: G\rightarrow \Aut((A\otimes B)_R)$. The cross product algebra $A\otimes_G B$ is the fixed subalgebra relative to this homomorphism $\rho \otimes \sigma$. Thus the cross product algebra $L_{\mathbb K}(Q_1) \otimes_G L_{\mathbb K}(Q_2)$ is the $\mathbb Z$-graded subalgebra of $L_{\mathbb K}(Q_1) \otimes_{\mathbb K} L_{\mathbb K}(Q_2)$ whose homogeneous component of degree $n$ is the linear span of the elements $x_n\otimes y_n$ where $x_n\in L_{\mathbb K}(Q_1)$ and $y_n\in L_{\mathbb K}(Q_2)$ are homogeneous of degree $n$.

\begin{theorem} \cite[Theorem 7]{GBG} 
Let $Q$ be a row-finite quiver with no sink. Then the cross product algebra $L_{\mathbb K}(Q) \otimes_G L_{\mathbb K}(Q)$ is isomorphic to $L_{\mathbb K}(\widehat{Q})$.
\end{theorem}

\noindent In \cite{GBG} it is shown that the assumption of ``no sink" cannot be dropped in showing that $L_{\mathbb K}(Q_1) \otimes_G L_{\mathbb K}(Q_2) \cong L_{\mathbb K}(Q_1\otimes Q_2)$, where $Q_1\otimes Q_2$ is the Kronecker product. So, we proceed to consider the following question for the case when $Q_1$ and $Q_2$ are the same. 

\begin{question}\label{Q1}
For a finite quiver $Q$, is the cross product algebra $L_{\mathbb K}(Q) \otimes_G L_{\mathbb K}(Q)$ isomorphic to $L_{\mathbb K}(\widehat{Q})$?
\end{question}

\subsection{Some positive results}

\noindent Let $Q$ be a finite or infinite-line quiver. So $Q_0=\{e_i\}_{i\in I}$ with $I$, a subset of $\N\setminus\{0\}$ and there is only one edge from each $e_i$ to $e_{i+1}$. We can consider $\mathbb KQ$ to be the linear span of the $e_i$'s and define the
elements of the dual $\mathbb KQ^*$ given by
$e^i(e_j):=\delta_{i,j}$. Besides we consider the linear span of the tensors $e_i\otimes e^j\in \mathbb KQ\otimes (\mathbb KQ)^*$. Then $L_{\mathbb K}(Q)$ is (isomorphic to) the linear span of these elements
$$L_{\mathbb K}(Q)\cong\Span\{e_i\otimes e^j\colon i,j\in I\}.$$
Notice that multiplication of these tensors is given by the contractions
$(e_i\otimes e^j)(e_k\otimes e^l):=\delta_{j,k} e_i\otimes e^l$ which mimics the matrix multiplication. So, tensors $e_i\otimes e^i$ are idempotents, and any two of them are orthogonal. Of course, this algebra is isomorphic to $M_n(\mathbb K)$ where $n=\vert I\vert$ (so it could be $M_\infty(\mathbb K)$ when $I$ is infinite). Moreover, this algebra is
$\Z$-graded by declaring $e_i\otimes e^j$ to be of degree $\partial(e_i\otimes e^j):=j-i$. The vertices of $Q$ can be identified with the idempotents $e_i\otimes e^i$ and the edge from $e_i$ to $e_{i+1}$ with the tensor $e_i\otimes e^{i+1}$.

If $Q$ is a disjoint union of line quivers, or equivalently, if the connected components of $Q$ are line quivers $Q^a$ ($a\in A$), then we know that $L_{\mathbb K}(Q)=\oplus_{a\in A}L_{\mathbb K}(Q^a)$
and each $L_{\mathbb K}(Q^a)$ is isomorphic to the linear span of elements $e_i\otimes e^j$
with $i,j$ ranging in a suitable set of indices $I^a$. In this case $(e_i\otimes e^j)(e_k\otimes e^l)=0$ in case $i,j\in I^a$ and $k,l\in I^b$ with $a\ne b$.
Moreover, the degree map of $L_{\mathbb K}(Q)$ restricted to each summand $L_{\mathbb K}(Q^a)$ is
just the degree map of $L_{\mathbb K}(Q^a)$.

\begin{proposition}\label{gradediso}
    If $Q$ and $Q'$ are line quivers (finite or infinite), then we have a graded isomorphism $L_K(Q\otimes Q')\cong L_K(Q)\otimes_G L_K(Q')$.
\end{proposition}

\begin{proof}
Let $Q$ and $Q'$ be line quivers (finite or not) with $Q_0=\{u_i\}_{i\in I}$ and $Q'_0=\{v_j\}_{j\in J}$, then $(Q\otimes Q')_0=\{[u_i,v_j]\colon i\in I, j\in J\}$. Notice that $I,J\subset\N\setminus\{0\}$. 
Note that along one connected component of $Q\otimes Q'$ the vertices of this component are of the form $[u_i,v_j]$ with $j-i$ taking the same value all through the component. Thus, we can label the different connected components with this value so that
  $Q\otimes Q'=\sqcup_{a \in A}(Q \otimes Q')^a$, where for each $a \in A, \ (Q \otimes Q')^a$ is a connected component in $Q\otimes Q'$ and all the vertices $[u_i,v_j] \in (Q \otimes Q')^a$ satisfy that $j-i$ is constant. 

Then $L_{\mathbb K}((Q\otimes Q')^a)$ is the linear span of the tensors $[u_i,v_j]\otimes [u^k,v^l]$ with $j-i=l-k=a$. Of course, we have 

$$\begin{cases}L_{\mathbb K}((Q\otimes Q')^a)L_{\mathbb K}((Q\otimes Q')^b)=0, & (a\ne b)\\
L_{\mathbb K}((Q\otimes Q')^a)^2\subset L_{\mathbb K}((Q\otimes Q')^a). &\end{cases}$$
Now we define a graded homomorphism
$$\Omega_a\colon L_{\mathbb K}((Q\otimes Q')^a)\to L_{\mathbb K}(Q)\otimes_G L_{\mathbb K}(Q')$$ such that 
$$[u_i,v_j]\otimes[u^k,v^l]\mapsto (u_i\otimes u^k)\times(v_j\otimes v^l).$$
Note that since $j-i=l-k$ we also have 
$k-i=l-j$ which implies $(u_i\otimes u^k)\otimes(v_j\otimes v^l)\in L_{\mathbb K}(Q)\otimes_G L_{\mathbb K}(Q')$. 

It is easy to check that $\Omega_a$ is a  homomorphism of $\mathbb K$-algebras. Furthermore, each element of
$L_{\mathbb K}(Q)\otimes_G L_{\mathbb K}(Q')$ is in the image of a suitable $\Omega_a$. So, we have an
epimorphism $$\Omega\colon L_{\mathbb K}(Q\otimes Q')\to L_{\mathbb K}(Q)\otimes_G L_{\mathbb K}(Q')$$ whose
restriction to $L_{\mathbb K}((Q\otimes Q')^a)$ is $\Omega_a$. We need to check that $\Omega$ is a monomorphism. 

For this, consider the $\mathbb K$-bilinear map $$\Psi\colon L_{\mathbb K}(Q)\times L_{\mathbb K}(Q')\to L_{\mathbb K}(Q\otimes Q')$$ such that 
$\Psi(u_i\otimes u^j,v_k\otimes v^l)=[u_i,v_k]\otimes[u^j,v^l]$. 

This induces a $\mathbb K$-linear map 
$$\Phi\colon L_{\mathbb K}(Q)\otimes L_{\mathbb K}(Q')\to L_{\mathbb K}(Q\otimes Q')$$ such that 
$(u_i\otimes u^j)\otimes(v_k\otimes v^l)\mapsto [u_i,v_k]\otimes[u^j,v^l]$. 

The restriction of this map to $L_{\mathbb K}(Q)\otimes_G L_{\mathbb K}(Q')$ satisfies
$$\Phi\Omega_a([u_i,v_j]\otimes[u^k,v^l])=
\Phi((u_i\otimes u^k)\otimes(v_j\otimes v^l))=[u_i,v_j]\otimes[u^k,v^l].$$ This shows $\Omega$ is a monomorphism. Thus we have $L_{\mathbb K}(Q\otimes Q')\cong L_{\mathbb K}(Q)\otimes_G L_{\mathbb K}(Q')$.
\end{proof}

In particular, we have 
\begin{corollary}
If $Q$ is finite or infinite line quiver, then we have a graded isomorphism $L_{\mathbb K}(\widehat{Q})\cong L_{\mathbb K}(Q)\otimes_G L_{\mathbb K}(Q)$.
\end{corollary}

\subsection{A counterexample for locally finite Leavitt path algebras }

\noindent A $\mathbb{Z}$-graded ${\mathbb K}$-algebra $A=\oplus_{n \in \mathbb{Z}} A_n$
is \textit{locally finite} if $\dim_{\mathbb K}(A_n)$ is finite for all $n \in \mathbb{Z}$. Recall that the Leavitt path algebra $L_{\mathbb K}(Q)$ is locally finite if and only if $Q$ is finite and no cycle in $Q$ has an exit. Consider the following quiver:

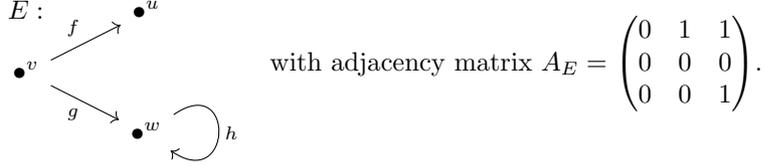
\[
\hbox{$\begin{tikzcd}[row sep=small, column sep= small]E:
	&& \bullet^{u} \\
	\bullet^{v} \\
	&& \bullet^{w}
	\arrow[from=2-1, to=1-3, "f"]
	\arrow[from=2-1, to=3-3, "g"']
	\arrow[from=3-3, to=3-3, loop, in=325, out=35, distance=10mm, "h"]
\end{tikzcd}$\quad  with adjacency matrix $A_E=\begin{pmatrix}0 & 1 & 1\\0 & 0 & 0\\ 0 & 0 & 1\end{pmatrix}$}.
\]
Then the Kronecker square  $\widehat{E}$ is as follows:

\bigskip 
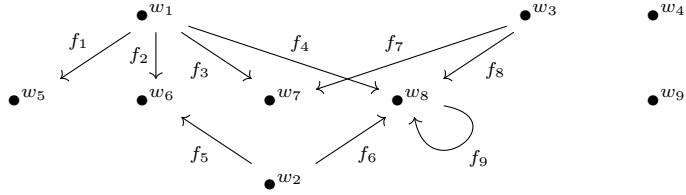
\[\begin{tikzcd}
	& \bullet^{w_1} &&& \bullet^{w_3} & \bullet^{w_4} \\
	\bullet^{w_5} & \bullet^{w_6} & \bullet^{w_7} & \bullet^{w_8} && \bullet^{w_9} \\
	&& \bullet^{w_2}
	\arrow[from=1-2, to=2-1, "f_1"']
	\arrow[from=1-2, to=2-2, "f_2"']
	\arrow[from=1-2, to=2-3, "f_3"']
	\arrow[from=1-2, to=2-4, "f_4"]
	\arrow[from=1-5, to=2-3, "f_7"']
	\arrow[from=1-5, to=2-4, "f_8"]
	\arrow[from=2-4, to=2-4, loop, in=280, out=350, distance=10mm, "f_9"]
	\arrow[from=3-3, to=2-2, "f_5"]
	\arrow[from=3-3, to=2-4, "f_6"']
\end{tikzcd}\]

whose adjacency matrix is 
$$\small
\left(
\begin{array}{ccccccccc}
 0 & 0 & 0 & 0 & 1 & 1 & 1 & 1 & 0 \\
 0 & 0 & 0 & 0 & 0 & 1 & 0 & 1 & 0 \\
 0 & 0 & 0 & 0 & 0 & 0 & 1 & 1 & 0 \\
 0 & 0 & 0 & 0 & 0 & 0 & 0 & 0 & 0 \\
 0 & 0 & 0 & 0 & 0 & 0 & 0 & 0 & 0 \\
 0 & 0 & 0 & 0 & 0 & 0 & 0 & 0 & 0 \\
 0 & 0 & 0 & 0 & 0 & 0 & 0 & 0 & 0 \\
 0 & 0 & 0 & 0 & 0 & 0 & 0 & 1 & 0 \\
 0 & 0 & 0 & 0 & 0 & 0 & 0 & 0 & 0 \\
\end{array}
\right).
$$
Now, we can compute the Grothendieck group $K_0(L_{\mathbb K}(\widehat{E}))$. Let us first recall the basics of Grothendieck group. For any ring $R$, let $\nu(R)$ denote the monoid of isomorphism classes of finitely generated right projective modules over $R$ with direct sum as the addition operation. For a finitely generated projective $R$-module $P$, we denote the isomorphism class of $P$ by $[P]$. For $[P], [Q] \in \nu(R)$, we have $[P]+[Q]=[P\oplus Q]$. The Grothendieck group, $K_0(R)$, is defined as the group completion of the monoid $\nu(R)$. It is known that if $A$ and $B$ are isomorphic algebras then $K_0(A)$ and $K_0(B)$ are isomorphic as groups. 

We use software to compute the $K_0$ group from the adjacency matrix, namely the one used in \cite{Atlas}, and the Grothendieck group $K_0(L_{\mathbb K}(\widehat{E}))$ turns out to be $\Z^6$.

By using out-split we can get a similar example. We can consider the out-split graph constructed from $E$ following \cite[Definition 2.6]{ABRAMS20081983} for the partition $s^{-1}(v)=\varepsilon_v^{1}\sqcup\varepsilon_v^{2}$ with $\varepsilon_v^{1}=\{ f\}$, $\varepsilon_v^{2}=\{ g\}$ and $s^{-1}(w)=\varepsilon_w^{1}=\{ h\}$. So, we have $m(v)=2, m(u)=0$ and $m(w)=1$. Furthermore,  $E_s({\mathcal P})^0=\{v_1, v_2, w_1, u\}$ and $E_s({\mathcal P})^1=\{f, g_1, h_1\}$ with the resulting out-split graph $E'$:

\[\begin{tikzcd}E':  
	\bullet^{v_1} && \bullet^{u} \\
	\bullet^{v_2} && \bullet^{w_1}
	\arrow[from=1-1, to=1-3, "f"]
	\arrow[from=2-1, to=2-3, "g_1"]
	\arrow[from=2-3, to=2-3, loop, in=325, out=35, distance=10mm, "h_1"]
\end{tikzcd}\]
By \cite[Theorem 2.8]{ABRAMS20081983} $L_K(E)$ is graded isomorphic to $L_K(E')$. The Kronecker square $\widehat{E'}$ is as follows:

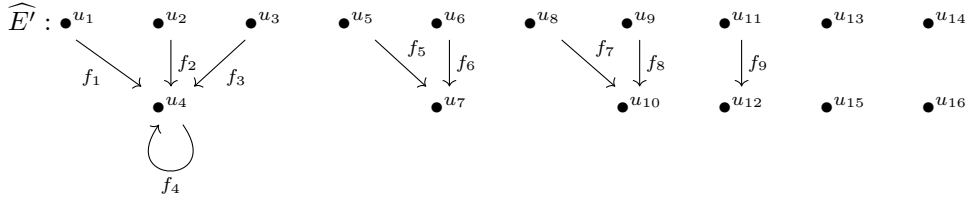
\[\begin{tikzcd}[column sep=small] \widehat{E'}:
	\bullet^{u_1} & \bullet^{u_2} & \bullet^{u_3} & \bullet^{u_5} & \bullet^{u_6} & \bullet^{u_8} & \bullet^{u_9} & \bullet^{u_{11}} & \bullet^{u_{13}} & \bullet^{u_{14}} \\
	& \bullet^{u_4} &&& \bullet^{u_7} && \bullet^{u_{10}} & \bullet^{u_{12}} & \bullet^{u_{15}} & \bullet^{u_{16}}
	\arrow[from=1-1, to=2-2, "f_1"']
	\arrow[from=1-2, to=2-2,"f_2"]
	\arrow[from=1-3, to=2-2, "f_3"]
	\arrow[from=1-4, to=2-5, "f_5"]
	\arrow[from=1-5, to=2-5, "f_6"]
	\arrow[from=1-6, to=2-7, " f_7"]
	\arrow[from=1-7, to=2-7, "f_8"]
	\arrow[from=1-8, to=2-8, " f_9"]
	\arrow[from=2-2, to=2-2, loop, in=235, out=305, distance=10mm, "f_4"]
\end{tikzcd}\]
whose adjacency matrix for a convenient reorganization of the vertices is   
$$\small\left(
\begin{array}{cccccccccccccccc}
 0 & 0 & 0 & 0 & 0 & 1 & 0 & 0 & 0 & 0 & 0 & 0 & 0 & 0 & 0 & 0 \\
 0 & 0 & 0 & 0 & 0 & 0 & 0 & 0 & 0 & 0 & 0 & 0 & 0 & 0 & 0 & 0 \\
 0 & 0 & 0 & 0 & 0 & 0 & 0 & 1 & 0 & 0 & 0 & 0 & 0 & 0 & 0 & 0 \\
 0 & 0 & 0 & 0 & 0 & 0 & 0 & 1 & 0 & 0 & 0 & 0 & 0 & 0 & 0 & 0 \\
 0 & 0 & 0 & 0 & 0 & 0 & 0 & 0 & 0 & 0 & 0 & 0 & 0 & 0 & 0 & 0 \\
 0 & 0 & 0 & 0 & 0 & 0 & 0 & 0 & 0 & 0 & 0 & 0 & 0 & 0 & 0 & 0 \\
 0 & 0 & 0 & 0 & 0 & 0 & 0 & 0 & 0 & 0 & 0 & 0 & 0 & 0 & 0 & 0 \\
 0 & 0 & 0 & 0 & 0 & 0 & 0 & 0 & 0 & 0 & 0 & 0 & 0 & 0 & 0 & 0 \\
 0 & 0 & 0 & 0 & 0 & 0 & 0 & 0 & 0 & 0 & 0 & 0 & 0 & 1 & 0 & 0 \\
 0 & 0 & 0 & 0 & 0 & 0 & 0 & 0 & 0 & 0 & 0 & 0 & 0 & 0 & 0 & 0 \\
 0 & 0 & 0 & 0 & 0 & 0 & 0 & 0 & 0 & 0 & 0 & 0 & 0 & 0 & 0 & 1 \\
 0 & 0 & 0 & 0 & 0 & 0 & 0 & 0 & 0 & 0 & 0 & 0 & 0 & 0 & 0 & 1 \\
 0 & 0 & 0 & 0 & 0 & 0 & 0 & 0 & 0 & 0 & 0 & 0 & 0 & 1 & 0 & 0 \\
 0 & 0 & 0 & 0 & 0 & 0 & 0 & 0 & 0 & 0 & 0 & 0 & 0 & 0 & 0 & 0 \\
 0 & 0 & 0 & 0 & 0 & 0 & 0 & 0 & 0 & 0 & 0 & 0 & 0 & 0 & 0 & 1 \\
 0 & 0 & 0 & 0 & 0 & 0 & 0 & 0 & 0 & 0 & 0 & 0 & 0 & 0 & 0 & 1 \\
\end{array}
\right).$$

The $K_0$ group of $L_{\mathbb K}(\widehat{E'})$ turns out to be $\Z^8$.
Observe that both Leavitt path algebras $L_{\mathbb K}(E)$ and $L_{\mathbb K}(E')$ are locally finite algebras and since 
\begin{equation}\label{K0}
\begin{cases}
    K_0(L_{\mathbb K}(\widehat E))=\Z^6 \text{ and}\\
    K_0(L_{\mathbb K}(\widehat E'))=\Z^8,
\end{cases}
\end{equation}
we have $$L_{\mathbb K}(\widehat{E})\not\cong L_{\mathbb K}(\widehat{E'}).$$
But since $L_{\mathbb K}(E)\cong_{\text{gr}} L_{\mathbb K}(E')$
we have $L_{\mathbb K}(E)\otimes_G L_{\mathbb K}(E)\cong_{\text{gr}}L_{\mathbb K}(E')\otimes_G L_{\mathbb K}(E')$ which implies that 
either $L_{\mathbb K}(\widehat E)\not\cong L_{\mathbb K}(E)\otimes_G L_{\mathbb K}(E)$ or $L_{\mathbb K}(\widehat E')\not\cong L_{\mathbb K}(E')\otimes_G L_{\mathbb K}(E')$. So, we deduce that for locally finite Leavitt path algebras the Question \ref{Q1} has a negative answer.

To go a little deeper into the details, we can see that 
$$L_{\mathbb K}(E)_0=\langle u, w\rangle + \langle ff^*, gg^*, gh^*, hg^*\rangle,$$ so $\dim(L_{\mathbb K}(E)_0)=6$ and $\dim((L_{\mathbb K}(E)\otimes_G L_{\mathbb K}(E))_0)=\dim(L_{\mathbb K}(E)_0\otimes_G L_{\mathbb K}(E)_0)=36$. If  we consider the sets $D_1=\{f_1f_1^*, f_2f_2^*, f_2f_5^*, f_3f_3^*, f_3f_7^*, f_4f_6^*, f_4f_8^*, f_4f_9^*\}$ and $$D_2=\{ f_5f_2^*, f_6f_6^*, f_6f_4^*, f_6f_8^*, f_6f_9^*, f_7f_7^*, f_7f_3^*, f_8f_4^*, f_8f_6^*,  f_8f_9^*,f_9f_4^*, f_9f_6^*, f_9f_8^*\},$$ then $L_{\mathbb K}(\widehat{E})_0=\langle \{w_i\}_{i=1}^9 \cup D_1 \cup D_2 \rangle$ and $\dim(L_{\mathbb K}(\widehat{E})_0)=30$. Therefore, $L_{\mathbb K}(\widehat{E})\not\cong_{\text{gr}} L_{\mathbb K}(E)\otimes_G L_{\mathbb K}(E)$ in this case.
\newline
Also, we can observe that $\dim(L_{\mathbb K}(\widehat{E'})_0)=32$, since $$L_{\mathbb K}(\widehat{E'})_0=\langle \{u_i\}_{i=1}^{16} \cup \{f_if_j^*\}_{(i,j) \in S\cup S' }\rangle,$$ with $S=\{(1,2),(1,3), (1,4), (2,3), (2,4), (3,4), (5,6), (7,8)\}$ and $S'=\{ (i,j)\colon (j, i) \in S \}$. Therefore, there is no graded isomorphism between $L_{\mathbb K}(E')\otimes_G L_{\mathbb K}(E')$
and $L_{\mathbb K}(\widehat{E'})$.

 \subsection{A formulation of the conjecture}
 Recall that a vertex $v\in Q_0$ is a bifurcation (or that there is a bifurcation at $v$) if $s^{-1}(v)$ has at least two elements. A vertex $u\in Q_0$ is called a \textit{line point} if there are neither bifurcations nor cycles at any vertex $w \in T(u)$. For a quiver $Q$, denote by $\pl(Q)$ the set of point-line vertices of the quiver $Q$.  We know that $\soc(L_{\mathbb K}(Q))$ is the ideal of $L_{\mathbb K}(Q)$ generated by
 $\pl(Q)$. This allows us to define a sequence of vertices $\pl^n(Q)$ in the following way: we write $\pl^1(Q):=\pl(Q)$ and for any ordinal:
 $$\begin{cases}\pl\nolimits^{\a+1}(Q)=\pl\nolimits^\a(Q)\cup\pl(Q\setminus\overline{\pl\nolimits^\a(Q)}),\\
 \pl\nolimits^\b(Q)=\cup_{\a<\b}\pl\nolimits^\a(Q), \text{if $\b$ is a limit ordinal},\end{cases}$$
 
\noindent where $\overline{\pl\nolimits^\a(Q)}$ denotes the  hereditary saturated closure of $\pl\nolimits^\a(Q)$.
 
 In the setting of finite quivers we will need the sets $\pl^n(Q)$
 only for finite ordinals $n$. Recall that for an algebra $A$, the Loewy series of socles $\soc^\a(A)$ is defined analogously by $\soc^1(A):=\soc(A)$ and for any ordinal
 $$\begin{cases}\soc(A/\soc^\a(A))=\soc^{\a+1}(A)/\soc^\a(A),\\
 \soc^\b(A)=\sum_{\a<\b}\soc^\a(A), \text{if $\b$ is a limit ordinal}.\end{cases}$$
 \begin{definition}
 We define $$\overrightarrow{\pl}(Q):=\cup_\a\pl\nolimits^\a(Q), \quad 
 \overrightarrow{\soc}(L_{\mathbb K}(Q)):=\sum_\a\soc\nolimits^\a(L_{\mathbb K}(Q))$$
 where $\a$ ranges in the class of all ordinals.
 \end{definition}
 
 \begin{conjecture} \label{conj1} There exists a short exact sequence
 \
 \begin{equation}\label{new}
 \overrightarrow{\soc}(L_{\mathbb K}(\widehat{Q}))\hookrightarrow L_{\mathbb K}(\widehat{Q})\twoheadrightarrow L_{\mathbb K}(Q\setminus \overrightarrow{P_l})\otimes_G L_{\mathbb K}(Q\setminus \overrightarrow{P_l}),
 \end{equation}
where $Q\setminus\overrightarrow{\pl}$ denotes the quiver obtained by
removing from $Q$, the hereditary and saturated closure of $\overrightarrow{\pl}$. 
%(\color{red} We know that $\overrightarrow{\pl}$ is hereditary. Maybe $\overrightarrow{\pl}$ is also saturated in which case we should not need the mentioned closure).
\end{conjecture}

  Notice that in case $P_l=\emptyset$ (which implies that there are no sinks), formula \eqref{new}  reduces to $L_{\mathbb K}(\widehat{Q})\cong L_{\mathbb K}(Q)\otimes_G L_{\mathbb K}(Q)$ which is precisely \cite[Theorem 7]{GBG}. When $Q=P_l$, it reduces to a well-known fact that $\soc(L_{\mathbb K}(\widehat{Q}))=L_{\mathbb K}(\widehat{Q})$.
Observe that the short exact sequence \eqref{new} is equivalent to the isomorphism 
$$\frac{L_{\mathbb K}(\widehat{Q})}{\overrightarrow{\soc}(L_{\mathbb K}(\widehat{Q}))}\cong L_{\mathbb K}(Q\setminus \overrightarrow{P_l})\otimes_G L_{\mathbb K}(Q\setminus \overrightarrow{P_l})$$

Next, we proceed to show that if a quiver $Q$ has finite number of vertices, then the set $\pl(Q\setminus\overline{\pl(Q)})$ is an empty set. 

\begin{proposition}\label{alme} Let $Q$ be a row-finite quiver.
    If $u\in \pl (Q\setminus{\overline{\pl(Q)}})$, then there is an $f \in s^{-1}(u)$ with $r(f)\notin\overline{\pl(Q)}$ and $r(f)\neq u$.
    More generally if $\{u_1,\ldots,u_k\}\cap\overline{\pl(Q)}=\emptyset$
    and the $u_j$'s are all different and satisfy the following two conditions:
    \begin{enumerate}
        \item $u_{i+1}\in r(s^{-1}(u_i))$ for $i=1,\ldots ,k-1$ and
        \item $u_1\in\pl(Q\setminus\overline{\pl(Q)}),$
        \end{enumerate} then there is an
    $u_{k+1}\not\in\{u_1,\ldots,u_k\}$ such that 
    $u_{k+1}\not\in\overline{\pl(Q)}$ and $u_{k+1}\in r(s^{-1}(u_k))$.
\end{proposition}
\begin{proof}
If $r(s^{-1}(u))\subset \overline{\pl(Q)}$, then $u \in \overline{\pl(Q)}$, a contradiction. Therefore, there is some edge $f\in s^{-1}(u)$ such that $r(f)\not\in\overline{\pl(Q)}$. Besides, if $r(f)=u$, then the tree of $u$ in the quiver $Q\setminus\overline{\pl(Q)}$ contains the cycle $f$ which is a contradiction to the fact that $u$ is a line-point of this graph.
For the second assertion, we apply the previously proved result to $u_k$. So we need to check that $u_k \in\pl (Q\setminus{\overline{\pl(Q)}})$, but $u_k$ is in the tree of $u_1$ relative to the graph $Q\setminus \overline{\pl (Q)}$.  So, being $u_1$ a point-line, then $u_k$ also is.
Consequently, there is an edge $f$ whose
range (denoted $u_{k+1}$) is not in $\overline{\pl(Q)}$ and $u_{k+1}\ne u_k$. Moreover,
if $u_{k+1}=u_j$ with $j<k$, then there is a closed path whose vertices are $u_j,\ldots,u_k,u_{k+1}$ and this contradicts that $u_1\in\pl(Q\setminus\overline{\pl(Q)}).$
\end{proof}
As a consequence we have the following.
\begin{corollary}
If for a row-finite quiver $Q$, the set of vertices $Q_0$ is finite, then we have $\pl(Q\setminus\overline{\pl(Q)})=\emptyset$. Furthermore, $\overrightarrow{\pl}(Q)=P_l^1(Q)$ and $\soc^{2}(L_{\mathbb K}(Q))=\soc^{1}(L_{\mathbb K}(Q))$ implying $\overrightarrow{\soc}(L_{\mathbb K}(Q))=\soc^{1}(L_{\mathbb K}(Q))$.

\end{corollary}

\begin{proof}
If $\pl(Q\setminus\overline{\pl(Q)})$ is non-empty then by the Proposition \ref{alme}, the set of vertices $Q_0$ is infinite. Thus, if $Q_0$ is finite, then we have $\pl(Q\setminus\overline{\pl(Q)})=\emptyset$. Therefore $\pl^2(Q)=\pl^1(Q)\cup \pl(Q\setminus\overline{\pl(Q)})=\pl^1(Q)$, so it follows $\overrightarrow{\pl}(Q)=P_l^1(Q)$. On the other hand, $\soc^2(L_{\mathbb K}(Q))/\soc^1(L_{\mathbb K}(Q))=\soc(L_{\mathbb K}(Q)/\soc^1(L_{\mathbb K}(Q)))=\soc(L_{\mathbb K}(Q\setminus\overline{\pl(Q)}))$. Since $\soc(L_{\mathbb K}(Q\setminus\overline{\pl(Q)}))$ is the ideal generated by $\pl(Q\setminus \overline{\pl(Q)})$, we have $\soc(L_{\mathbb K}(Q\setminus\overline{\pl(Q)}))=0$ and so $\soc^2(L_{\mathbb K}(Q))=\soc^1(L_{\mathbb K}(Q))$ hence
$\overrightarrow{\soc}(L_{\mathbb K}(Q))=\soc^{1}(L_{\mathbb K}(Q))$.
\end{proof}

Thus, for a finite quiver $Q$, Conjecture \ref{conj1} adopts the following form.
\begin{conjecture} \label{conj2}
For a finite quiver $Q$, there exists a short exact sequence
\begin{equation}
 \soc(L_{\mathbb K}(\widehat{Q}))\hookrightarrow L_{\mathbb K}(\widehat{Q})\twoheadrightarrow L_{\mathbb K}(Q\setminus \overline{P_l(Q)})\otimes_G L_{\mathbb K}(Q\setminus \overline{P_l(Q)}),
 \end{equation}
or equivalently
\begin{equation}\label{alme2}
    \frac{L_{\mathbb K}(\widehat{Q})}{\soc(L_{\mathbb K}(\widehat{Q}))}\cong L_{\mathbb K}(Q\setminus \overline{\pl(Q))}\otimes_G L_{\mathbb K}(Q\setminus \overline{\pl(Q)}).
\end{equation}
\end{conjecture}
\begin{remark} \rm
Let us see that the Toeplitz algebra satisfies the above conjecture. We will denote by $T$ the quiver 

\[T: \begin{tikzcd}
	 {\bullet_u} & {\bullet_v}
	\arrow[from=1-1, to=1-1, loop, in=145, out=215, distance=10mm]
	\arrow[from=1-1, to=1-2]
\end{tikzcd}\]
Here we have $\pl(T)=\overline{\pl(T)}=\{v\}$ and 
the quiver $T\setminus\overline{\pl(T)}=T\setminus\{v\}$ is a simple loop. 
The adjacency matrix of $T$ is $\tiny\begin{pmatrix}1 & 1\\0 & 0\end{pmatrix}$. We have 
$$\tiny\begin{pmatrix}1 & 1\\0 & 0\end{pmatrix}\otimes \begin{pmatrix}1 & 1\\0 & 0\end{pmatrix}=\begin{pmatrix}1 & 1 & 1 & 1\\
0 & 0 & 0 &0\\0 & 0 & 0 &0\\0 & 0 & 0& 0\\
\end{pmatrix}.$$
Then the Kronecker square $\widehat{T}$ is the quiver:

\[\widehat{T}: \begin{tikzcd}
	&& \bullet_{v_1} \\
	\bullet_u && \bullet_{v_2} \\
	&& \bullet_{v_3}
	\arrow[from=2-1, to=1-3,"f_1"]
	\arrow[from=2-1, to=2-1, "c", loop, in=145, out=215, distance=10mm]
	\arrow[from=2-1, to=2-3,"f_2"]
	\arrow[from=2-1, to=3-3,"f_3"]
\end{tikzcd}\]
So, $\soc(L_{\mathbb K}(\widehat{T}))\cong M_{\infty}({\mathbb K})^3$ and $L_{\mathbb K}(\widehat{T})/M_{\infty}({\mathbb K})^3 \cong {\mathbb K}[x,x^{-1}]$.
On the other hand, $\pl(\widehat{T})=\overline{\pl(\widehat{T})}=\{v_1,v_2,v_3\}$ and $\widehat{T}\setminus\overline{\pl(\widehat{T})}=\widehat{T}\setminus\{v_1,v_2,v_3\}$ is
a simple loop. So, in equation \eqref{alme2}, 
$L_{\mathbb K}(T\setminus \overline{\pl(T)})\otimes_G L_{\mathbb K}(T\setminus \overline{\pl(T)})\cong {\mathbb K}[x,x^{-1}]\otimes_G {\mathbb K}[x,x^{-1}]\cong {\mathbb K}[x,x^{-1}]$
(take into account that if $L$ is a Leavitt path algebra over a quiver $Q$ consisting of just one vertex with a single loop and $\widehat{L}$ is the Leavitt path algebra over $\widehat{Q}$ then we have $\widehat{L}\cong L$). This shows that the Conjecture \ref{conj2} holds for the Toeplitz algebra. Equivalently, we have a short exact sequence 
$$M_\infty({\mathbb K})^3\hookrightarrow L_{\mathbb K}(\widehat{T})\twoheadrightarrow {\mathbb K}[x,x^{-1}].$$

\end{remark}

\section{Leavitt path algebra over Kronecker square and Hayashi's Face Algebra}

\noindent Given a quiver $Q$, we extend it to a new quiver $\overline{Q}$, where for each arrow $e\in Q_{1}$, we have the ghost arrow $e^{\ast}$ in the opposite direction. We call the new quiver $\overline{Q}$ the \textit{double quiver}. 

\begin{definition} \rm
Let $Q=(Q_{0}, Q_{1}, s, r)$ be a finite quiver, $\overline{Q}$ be its double quiver and $\mathbb K$ be a field. We define an algebra $\mathcal{A}_{\mathbb K}(Q)$ to be the quotient algebra of face algebra $\mathcal{H}_{\mathbb K}(\overline{Q})$ over $\overline{Q}$ by the two-sided ideal $I$ generated by the following elements:

\medskip

(1) $x_{p^*, q^*} x_{p', q'} -\delta_{p, p'} \delta_{q, q'} x_{r(p), r(q)}$ for each $p, q, p', q' \in Q_1$. 

\medskip

(2) $\Sigma_{i=s(p), j=s(q)} x_{p, q} x_{p^*, q^*}-x_{i, j}$ for each $i, j\in Q_0$ that are not sinks, that is, they emit at least one arrow.   

\end{definition}

Next, we show that the algebra $\mathcal{A}_{\mathbb K}(Q)$ is indeed isomorphic to the Leavitt path algebra $ L_{\mathbb K}(\widehat{Q})$. 

\begin{theorem} \label{t1}
For a finite quiver $Q$, $\mathcal{A}_{\mathbb K}(Q) \cong L_{\mathbb K}(\widehat{Q})$.
\end{theorem}

\begin{proof}
Since $\{x_{v, w}, x_{p, q}, x_{p^*, q^*}: v, w\in Q_0, p, q \in Q_1\}$ is a set of generators in $\mathcal{A}_{\mathbb K}(Q)$ satisfying the CK-1 and CK-2 relations, by the universal property of Leavitt path algebras, there exists a unique $\mathbb K$-algebra homomorphism $\varphi: L_{\mathbb K}(\widehat{Q}) \rightarrow \mathcal{A}_{\mathbb K}(Q)$ such that $\varphi([v, w])=x_{v, w}$, $\varphi([p, q])=x_{p, q}$, and $\varphi([p, q]^{\ast})=x_{p^*, q^*}$. 
Note that $\mathcal{A}_{\mathbb K}(Q)$ is a $\mathbb Z$-graded algebra under the grading $\deg(x_{v, w})=0$, $\deg(x_{p, q})=1$, and $\deg(x_{p^*, q^*})=-1$. Clearly, $\varphi$ is graded and $\varphi$ is nonzero on each vertex, therefore, by the Graded Uniqueness Theorem, it follows that $\varphi$ is injective. Clearly this map $\varphi$ is surjective as well. Thus we have that $\mathcal{A}_{\mathbb K}(Q) \cong L_{\mathbb K}(\widehat{Q})$. 
\end{proof}

Next, we show that since the module category of Leavitt path algebras is well studied, it might help us understand the module category of Hayashi's face algebra. Recall that for a $\mathbb Z$-graded noncommutative $\mathbb K$-algebra $A$, the category of quasi-coherent sheaves on noncommutative scheme is defined as $\QGr A:= \Gr A/\Fdim A$ where $\Gr A$ is the category of $\mathbb Z$-graded modules over $A$ and $\Fdim A$ is the Serre subcategory of finite-dimensional submodules.  

Let $\mathbb K$ be a field, and $Q$ be a finite quiver. We denote the category of left $\mathbb KQ$-modules with degree-preserving homomorphisms as $\Gr\mathbb KQ$ and we denote its full subcategory of modules that are the sum of their finite-dimensional submodules. Smith \cite[Theorem 1.3]{Smith} studied the quotient category $\QGr\mathbb KQ:=\Gr\mathbb KQ/\Fdim\mathbb KQ$ and proved that if $Q^0$ is the quiver without sources or sinks that is obtained by repeatedly removing all sinks and sources from $Q$, then we have the category equivalence $\QGr\mathbb K Q \equiv \Gr L_{\mathbb K}(Q^0) \equiv \Mod (L_{\mathbb K}(Q^0))_0$.

First, we have the following simple observation that we state without proof. 
\begin{lemma}
For a finite quiver $Q$, let $Q^0$ denote the quiver without source or sinks obtained by repeatedly removing sources and sinks from $Q$. Then $(\widehat{Q})^0=\widehat{Q^0}$. 
\end{lemma}

A $\mathbb Z$-graded ring $R=\oplus_{n\in \mathbb Z} R_n$ is called \textit{strongly graded} if $R_nR_{-n}=R_0$ for each $n$. Let us denote by $\Gr R$, the category of graded right $R$-modules and by $\Mod R$, the category of right $R$-modules. Dade's Theorem  \cite[Theorem 2.8]{Dade} states that $R$ is strongly graded if and only if the functors $$F:\Gr R\rightarrow \Mod R_0$$
and $$G:\Mod R_0\rightarrow \Gr R$$ form mutually inverse equivalences of categories. 

\begin{proposition}
For any finite quiver $Q$, the following are equivalent:
\begin{enumerate}
\item $L_{\mathbb K}(Q)$ is strongly graded.
\item $L_{\mathbb K}(\widehat{Q})$ is strongly graded.
\item $\Gr L_{\mathbb K}(Q)\equiv \Mod (L_{\mathbb K}(Q))_0$.
\item $\Gr L_{\mathbb K}(\widehat{Q})\equiv \Mod (L_{\mathbb K}(\widehat{Q}))_0$.
\end{enumerate}
\end{proposition}

\begin{proof}
It is known that for a finite quiver $Q$, the Leavitt path algebra $L_{\mathbb K}(Q)$ is strongly graded if and only if $Q$ does not have a sink \cite{Hazrat}. As we have already observed that a quiver $Q$ has no sink if and only if its Kronecker square $\widehat{Q}$ has no sink. Thus it follows that $L_{\mathbb K}(\widehat{Q})$ is strongly graded if and only if $L_{\mathbb K}(Q)$ is strongly graded. The equivalence of categories in (3) and (4) then follows from Dade's Theorem. 
\end{proof}

The graded K-theory of Leavitt path algebras has been studied extensively in recent years. Let us briefly recall the basics of graded K-theory. Let $R$ be a graded unital ring and let $\nu^{gr}(R)$ denote the monoid of graded isomorphism classes of graded finitely generated right projective modules over $R$ with direct sum as the addition operation. For a graded finitely generated projective $R$-module $P$, we denote the graded isomorphism class of $P$ by $[P]$. For $[P], [Q] \in \nu^{gr}(R)$, we have $[P]+[Q]=[P\oplus Q]$. The graded Grothendieck group, $K_0^{gr}(R)$, is defined as the group completion of the monoid $\nu^{gr}(R)$. 

The next theorem shows that the graded Grothendieck group completely characterizes the quotient category of Hayashi's face algebra.   

\begin{theorem}
Let $Q$ and $Q'$ be finite quivers with no sinks. Then there is an ordered abelian group isomorphism $K_0^{gr} (\mathcal{A}_{\mathbb K}(Q)) \cong K_0^{gr} (\mathcal{A}_{\mathbb K}(Q'))$ if and only if there is a category equivalence $\QGr\mathcal H_{\mathbb K}(Q) \equiv \QGr\mathcal H_{\mathbb K}(Q')$. 
\end{theorem}

\begin{proof}
Since $Q$ and $Q'$ are quivers without sinks, the Kronecker squares $\widehat{Q}$ and $\widehat{Q'}$ have no sinks as well. Thus by Dade's Theorem it follows that $L_{\mathbb K}(\widehat{Q})$ and $L_{\mathbb K}(\widehat{Q'})$ are strongly graded. Consequently, as ordered abelian groups, $K_0^{gr} (L_{\mathbb K}(\widehat{Q})) \cong K_0^{gr} (L_{\mathbb K}(\widehat{Q'}))$ if and only if $K_0(L_{\mathbb K}(\widehat{Q})_0) \cong K_0(L_{\mathbb K}(\widehat{Q'})_0)$. Since $L_{\mathbb K}(\widehat{Q})_0$ and $L_{\mathbb K}(\widehat{Q'})_0$ are ultramatricial algebras, $K_0(L_{\mathbb K}(\widehat{Q})_0) \cong K_0(L_{\mathbb K}(\widehat{Q'})_0)$ as ordered abelian groups if and only if $L_{\mathbb K}(\widehat{Q})_0$ is Morita equivalent to $L_{\mathbb K}(\widehat{Q'})_0$. It is known that by repeatedly removing all the sources and sinks from $Q$ and $Q'$, we obtain new quivers $E$ and $E'$ without sources and sinks and hence the Kronecker squares $\widehat{E}$ and $\widehat{E'}$ are quivers without sources and sinks as well. By Smith \cite[Theorem 1.3]{Smith}, we know that $\Mod L_{\mathbb K}(\widehat{Q}) \equiv_{gr} \Mod L_{\mathbb K}(\widehat{E})$ and $\Mod L_{\mathbb K}(\widehat{Q'}) \equiv_{gr} \Mod L_{\mathbb K}(\widehat{E'})$. Consequently, we have $\Mod L_{\mathbb K}(\widehat{Q})_0 \equiv \Mod L_{\mathbb K}(\widehat{E})_0$ and $\Mod L_{\mathbb K}(\widehat{Q'})_0 \equiv \Mod L_{\mathbb K}(\widehat{E'})_0$. By Smith \cite[Theorem 1.3]{Smith}, we know that $\QGr \mathbb K \widehat{Q} \equiv \Mod L_{\mathbb K}(\widehat{E})_0$ and $\QGr \mathbb K \widehat{Q'} \equiv \Mod L_{\mathbb K}(\widehat{E'})_0$. Since $\mathbb K \widehat{Q} \cong \mathcal H_{\mathbb K}(Q)$ and $\mathbb K \widehat{Q'} \cong \mathcal H_{\mathbb K}(Q')$, it follows that $\QGr \mathbb K \widehat{Q} \equiv \QGr \mathcal H_{\mathbb K}(Q)$ and $\QGr \mathbb K \widehat{Q'} \equiv \QGr \mathcal H_{\mathbb K}(Q')$. Consequently, in view of the Theorem \ref{t1}, we have $K_0^{gr} (\mathcal{A}_{\mathbb K}(Q)) \cong K_0^{gr} (\mathcal{A}_{\mathbb K}(Q'))$ if and only if $\QGr\mathcal H_{\mathbb K}(Q) \equiv \QGr\mathcal H_{\mathbb K}(Q')$.           
\end{proof}

Two nonnegative integer matrices $A$ and $B$ are called \textit{shift equivalent} if there exist nonnegative integer matrices $R$ and $S$ such that $A^l=RS$ and $B^l=SR$, for some $l \in \mathbb N$, and $AR=RB$ and $SA=BS$. In the lemma below we show that if two matrices are shift equivalent then their Kronecker squares are shift equivalent too. 

 \begin{lemma} \label{se}
If the matrices $A$ and $B$ are shift equivalent then the Kronecker squares $A\otimes A$ and $B\otimes B$ are shift equivalent too.
\end{lemma}

\begin{proof}
Suppose the matrices $A$ and $B$ are shift equivalent. Then there exist nonnegative matrices $R$ and $S$ such that $A^l=RS$ and $B^l=SR$, for some $l \in \mathbb N$, and $AR=RB$ and $SA=BS$. Now $(A\otimes A)^l=A^l \otimes A^l =RS \otimes RS=(R\otimes R)(S\otimes S)$. Take $R_1=R\otimes R$ and $S_1=S\otimes S$. Then $(A\otimes A)^l=R_1 S_1$ and $(B\otimes B)^l=B^l \otimes B^l=SR\otimes SR=(S\otimes S)(R\otimes R)=S_1 R_1$. Also, note that $(A\otimes A)R_1=(A\otimes A)(R\otimes R)=AR \otimes AR=RB \otimes RB=(R\otimes R)(B\otimes B)=R_1(B\otimes B)$ and $S_1(A\otimes A)=(S\otimes S)(A\otimes A)=SA\otimes SA=BS\otimes BS=(B\otimes B)(S\otimes S)=(B\otimes B)S_1$. This shows that $A\otimes A$ and $B\otimes B$ are shift equivalent.
\end{proof}

Hazrat conjectured the following.

\begin{conjecture}
Let $Q$ and $Q'$ be finite quivers with no sinks. Then the following are equivalent:
\begin{enumerate}
\item The Leavitt path algebras $L_{\mathbb K}(Q)$ and $L_{\mathbb K}(Q')$ are graded Morita equivalent.
\item There is an order preserving $\mathbb Z[x, x^{-1}]$-module isomorphism between $K_0^{gr} (L_{\mathbb K}(Q))$ and $K_0^{gr} (L_{\mathbb K}(Q'))$.
\item The adjacency matrices of $Q$ and $Q'$ are shift equivalent.
\end{enumerate}
\end{conjecture}

\noindent It is known that (1) implies (3) and it is also known that (2) and (3) are equivalent. But it is still open whether (3) implies (1).

As a consequence of Lemma \ref{se}, we have the following.

\begin{corollary}
If the adjacency matrices of quivers $Q$ and $Q'$ are shift equivalent, then the adjacency matrices of their Kronecker squares $\widehat{Q}$ and $\widehat{Q'}$ are also shift equivalent.
\end{corollary}

Using Hazrat's result \cite[Corollary 12]{Hazrat}, we can prove the following for finite quivers with no sinks.

\begin{theorem}
Let $Q$ and $Q'$ be finite quivers with no sinks. If $L_{\mathbb K}(Q)$ and $L_{\mathbb K}(Q')$ are graded Morita equivalent, then the adjacency matrices of $\widehat{Q}$ and $\widehat{Q'}$ are shift equivalent.
\end{theorem}

\begin{proof}
Let $Q$ and $Q'$ be finite quivers with no sinks. Suppose $L_{\mathbb K}(Q)$ and $L_{\mathbb K}(Q')$ are graded Morita equivalent. Then by \cite[Corollary 12]{Hazrat}, we know that the adjacency matrices of $Q$ and $Q'$ are shift equivalent and hence by the above corollary, we have that the adjacency matrices of their Kronecker squares $\widehat{Q}$ and $\widehat{Q'}$ are also shift equivalent.
\end{proof}

\begin{remark}
If Hazrat's conjecture turns out to be true, then this means that for finite quivers $Q$ and $Q'$ without sinks, if $L_{\mathbb K}(Q)$ and $L_{\mathbb K}(Q')$ are graded Morita equivalent, then the $L_{\mathbb K}(\widehat{Q})$ and $L_{\mathbb K}(\widehat{Q'})$ are graded Morita equivalent too. 
\end{remark}

Using Hazrat's result and our Lemma \ref{se}, we can prove the following for finite quivers without sinks. 

\begin{theorem}
Let $Q$ and $Q'$ be finite quivers with no sinks. If there is an order preserving $\mathbb Z[x, x^{-1}]$-module isomorphism between $K_0^{gr} (L_{\mathbb K}(Q))$ and $K_0^{gr} (L_{\mathbb K}(Q'))$, then there is an order preserving $\mathbb Z[x, x^{-1}]$-module isomorphism between $K_0^{gr} (L_{\mathbb K}(\widehat{Q}))$ and $K_0^{gr} (L_{\mathbb K}(\widehat{Q'}))$ as well. 
\end{theorem}

\begin{proof}
Let $Q$ and $Q'$ be finite quivers with no sinks. Suppose there is an order preserving $\mathbb Z[x, x^{-1}]$-module isomorphism between $K_0^{gr} (L_{\mathbb K}(Q))$ and $K_0^{gr} (L_{\mathbb K}(Q'))$. By \cite[Corollary 12]{Hazrat}, the adjacency matrices $A_Q$ and $A_{Q'}$ are shift equivalent. By Lemma \ref{se}, the matrices $A_Q \otimes A_Q$ and $A_{Q'} \otimes A_{Q'}$ are shift equivalent. Since the matrices $A_Q \otimes A_Q$ and $A_{Q'} \otimes A_{Q'}$ are adjacency matrices of quivers $\widehat{Q}$ and $\widehat{Q'}$, respectively, it follows from \cite[Corollary 12]{Hazrat}, that there is an order preserving $\mathbb Z[x, x^{-1}]$-module isomorphism between $K_0^{gr} (L_{\mathbb K}(\widehat{Q}))$ and $K_0^{gr} (L_{\mathbb K}(\widehat{Q'}))$.
\end{proof}

We can rephrase \cite[Theorem 26]{Hazrat} as follows.

\begin{theorem}
Let $Q$ be a finite quiver with no sink. Then there is an order preserving $\mathbb Z[x, x^{-1}]$-module isomorphism $$ K_0^{gr} (L_{\mathbb K}(\widehat{Q})) \cong K_0^{gr} (L_{\mathbb K}(Q)) \otimes_{\mathbb Z} K_0^{gr} (L_{\mathbb K}(Q)).$$
\end{theorem}

\section{Direct limits and Kronecker square}

\noindent Let us first recall the basics of Cuntz-Krieger quiver category from \cite{Goodearl}. Let $E=(E_0, E_1, s_E, r_E)$ and $F=(F_0, F_1, s_F, r_F)$ be two quivers. A quiver morphism from $E$ to $F$ is a pair $\varphi=(\varphi_0, \varphi_1)$ which consists of two maps $\varphi_0: E_0\rightarrow F_0$ and $\varphi_1: E_1\rightarrow F_1$ such that $s_F\varphi_1=\varphi_0s_E$ and $r_F\varphi_1=\varphi_0r_E$. We denote by $\bf Quiv$ the category of quivers where the objects are arbitrary quivers and the morphisms are quiver morphisms. We say that a quiver morphism $\varphi:E\rightarrow F$ is a $CK$-morphism if both $\varphi_0$ and $\varphi_1$ are injective and for each $v\in E_0$ which is neither a sink nor an infinite emitter, $\varphi_1$ induces a bijection $s^{-1}_E(v) \rightarrow s^{-1}_F(\varphi_0(v))$. We define $\bf CKQuiv$ to be the subcategory of $\bf Quiv$ whose objects are arbitrary quivers and morphisms are arbitrary $CK$-morphisms. 

The following is a simple observation.

\begin{lemma}\label{puthat}
    Let $f\colon E \to F$ be a $CK$-morphism in the category $\bf CKQuiv$ with $f=(f_0, f_1)$. Define $\hat{f}\colon \widehat{E} \to \widehat{F}$ with $\hat{f}=(\hat{f_0}, \hat{f_1})$ where $\hat{f_0}:\widehat{E_0}\rightarrow \widehat{F_0}$ given as $\hat{f_0}([v, w])=[{f_0}(v), {f_0}(w)]$ and $\hat{f_1}:\widehat{E_1}\rightarrow \widehat{F_1}$ given as $\hat{f_1}([e_1, e_2])=[{f_1}(e_1), {f_1}(e_2)]$. Then $\hat{f}$ is a $CK$-morphism in the same category.
\end{lemma}

\begin{proposition}\label{sombrero}
    Consider a directed set $(I, \le)$ and the direct system of quivers $$(\{E_i\}_{i \in I}, \{e_{ji}\}_{i\le j})$$ in the category $\bf CKQuiv$. If $E=\L E_i$, then $\widehat{E} =\L \widehat{E_i}$.
\end{proposition}
\begin{proof}
   Let  $(\{E_i\}_{i \in I}, \{e_{ji}\}_{i\le j})$ be a direct system of quivers in the category $\bf CKQuiv$ and suppose $E=\L E_i$, then we have the following commutative diagrams for $i\le j \le k$:
   \begin{center}
\begin{tikzcd}[cramped, column sep=small] 
E_i \arrow{r}{e_{ji}}  \arrow[rd,"e_{ki}"'] 
  & E_j \arrow{d}{e_{kj}} & E_i \arrow{r}{e_{ji}}\arrow[rd,"e_{i}"'] & E_j\arrow{d}{e_{j}}\\
    & E_k &  & E
\end{tikzcd}
\end{center}   
where the map $e_i$'s are  canonical. Using Lemma \ref{puthat}, we can consider the direct system of quivers $(\{\widehat{E_i}\}_{i \in I}, \{\hat{e}_{ji}\}_{i\le j})$, where $\hat{e}_{ji} \colon E_i\times E_i \to E_j \times E_j$ with $\hat{e}_{ji}(u,v)=(e_{ji}(u), e_{ji}(v))$ for all $(u,v) \in {E_{i0}}\times E_{i0}$ and  $\hat{e}_{ji}(f,g)=(e_{ji}(f), e_{ji}(g))$ for all $(f,g) \in E_{i1}\times E_{i1}$ and the following commutative diagrams for $i\le j \le k$:
 
 \begin{center}
\begin{tikzcd}[cramped, column sep=small] 
\widehat{E_i} \arrow{r}{\hat{e}_{ji}}  \arrow[rd,"\hat{e}_{ki}"'] 
  & \widehat{E_j} \arrow{d}{\hat{e}_{kj}} & \widehat{E_i} \arrow{r}{\hat{e}_{ji}}\arrow[rd,"\hat{e}_{i}"'] & \widehat{E_j}\arrow{d}{\hat{e}_{j}}\\
    & \widehat{E_k} &  & \widehat{E}
\end{tikzcd}
\end{center}

Suppose that we have another quiver $X$ and a family of morphism $\{t_i \}_{i\in I}$ with $t_i\colon \widehat{E_i} \to X$  in the category $\bf CKQuiv$, such that the following diagrams are commutative for $i \le j$:

\begin{center}
\begin{tikzcd}[cramped, column sep=small] 
  \widehat{E_i} \arrow{r}{\hat{e}_{ji}}\arrow[rd,"t_{i}"'] & \widehat{E_j}\arrow{d}{t_{j}}\\
    & X
\end{tikzcd}
\end{center}

We wish to prove that there exists a unique morphism $T$, making the following diagrams commutative.

\begin{center}
\begin{tikzcd}[cramped, column sep=small] 
  \widehat{E_i} \arrow{r}{\hat{e}_{i}}\arrow[rd,"t_{i}"'] & \widehat{E}\arrow{d}{T}\\
    & X
\end{tikzcd}
\end{center}

\noindent Consider $[u,v] \in \widehat{E}_0$, then $u, v\in E_0$ and there exist $i,j \in I$ such that $u=e_i(x_i)$ with $x_i \in E_i$ and $v=e_j(x_j)$ with $x_j \in E_j$. Furthermore, there exists $k \in I$ with $k\ge i,j$ and $e_{ki}(x_i), \ e_{kj}(x_j)\in E_k$. So, $e_ke_{ki}(x_i)=e_i(x_i)=u$ and $e_ke_{kj}(x_j)=e_j(x_j)=v$ and denoting by $z_k=e_{ki}(x_i)$ and ${z}_{k}'=e_{kj}(x_j)$, we have $u=e_k(z_k)$, $v=e_k(z_k')$ and $[u,v]=\hat{e}_k[z_k,z_k']$ with $[z_k,z_k']\in \widehat{E}_k$. We  define $T[u,v]:=t_k[z_k,z_k']$. Let us first show that it is well defined. Suppose that there exists $h \in I$ such that $u=e_h(r_h)$ and $v=e_h(r_h')$, then there exists $l \in I$, $l\ge h, k$ with $u=e_k(z_k)=e_le_{lk}(z_k)=e_h(r_h)=e_le_{lh}(r_h)$. Due to the injectivity of $e_l$ we have $e_{lk}(z_k)=e_{lh}(r_h)$ and analogously $e_{lk}(z_k')=e_{lh}(r_h')$. Therefore, $\hat{e}_{lk}[z_k,z_k']=\hat{e}_{lk}[r_h,r_h']$ and by the commutativity of the following diagram:

\begin{center}
\begin{tikzcd}[cramped, column sep=small] 
  \widehat{E_k} \arrow{r}{\hat{e}_{lk}}\arrow[rd,"t_{k}"'] & \widehat{E_l}\arrow{d}{t_{l}}\\
    & X
\end{tikzcd}
\end{center}
we have $t_k(z_k,z_k')=t_h(r_h,r_h')$. 
Note that an arrow $[e, g]\in \widehat{E_1}$ has $e, g\in E_1$. By the direct limit property of $E=\L E_i$ in $\bf CKQuiv$, there exists $k\in I$ and $a_k, b_k \in (E_i)_1$ for some $i\le k$ such that $e=e_k(a_k)$ and $g=e_k(b_k)$. Define $T[e, g]=t_k(a_k, b_k)$. It may be checked that $T$ is well defined on arrows as well. To check compatibility with $\hat{s}$, note $$\hat{s}_X(T[e, g])=\hat{s}_X(t_k[a_k, b_k])=t_k(\hat{s}_{\widehat E_k} [a_k, b_k])=t_k[s_{E_k}(a_k), s_{E_k}(b_k)]$$ and $$T(\hat{s}_{\widehat E}[e, g])=T[s_E(e), s_E(g)]=T[e_k(s_{E_k}(a_k)), e_k(s_{E_k}(b_k)]=t_k[s_{E_k}(a_k)), s_{E_k}(b_k)].$$
  Compatibility with $\hat{r}$ may be checked similarly. This shows $T$ is quiver morphism. Using the fact that each $t_k$ is a CK-morphism, with standard direct limit arguments it can be easily seen that $T$ is a CK-morphism. 
  For any vertex $[a, b]\in \widehat{E_{i}}$, we have $\hat{e_i}[a, b]=[e_i(a), e_i(b)] \in \widehat{E_0}$. Setting $u=e_i(a), v=e_i(b)$, the definition of $T$ gives $T[u, v]=t_i[a, b]$. So $T(\hat{e_i}[a, b])=t_i[a, b]$. We can verify similarly for arrows to conclude $T\circ \hat{e_i}=t_i$. Suppose there is another $T'$ satisfying $T'\circ \hat{e_i}=t_i$ for all $i$. For any $[u, v]\in \widehat{E_0}$, write $u=e_k(z_k)$, $v=e_k(z'_k)$. Then $T'[u, v]=T'(\hat{e_k}[z_k, z'_k])=t_k[z_k, z'_k]=T[u, v]$. We may verify similarly for arrows. This establishes the uniqueness of $T$. Thus we have  $\L \widehat{E_i}=\widehat{E}$.

\end{proof}
As a consequence, we have the following. 

\begin{corollary}\label{coro1} Consider a directed set $(I, \le)$ and the direct system of quivers $$(\{E_i\}_{i \in I}, \{e_{ji}\}_{i\le j})$$ in the category $\bf CKQuiv$. If $E=\L E_i$, then
$L_{\mathbb K}(\widehat{E})\cong \L L_{\mathbb K}(\widehat{E_i})$
\end{corollary}
\begin{proof}
Applying \cite[Lemma 2.5]{Goodearl} and the above proposition, we get $L_{\mathbb K}(\widehat{E})\cong L_{\mathbb K}(\L \widehat{E_i}) \cong\L L_{\mathbb K}(\widehat{E_i})$. 
\end{proof}

\noindent We denote the category of $\Z$-graded associative ${\mathbb K}$-algebras by $\Z\!\alg_{\mathbb K}$. If $A, B$ are objects of $\Z\!\alg_{\mathbb K}$; we define $\hom_{\Z\!\alg_{\mathbb K}}(A, B)$ as the set of all the graded homomorphisms of ${\mathbb K}$-algebras from $A$ to $B$. 

Consider the functor ${\mathcal G}\colon \Z\!\alg_{\mathbb K} \to \Z\!\alg_{\mathbb K}$ defined by ${\mathcal G}(A):=A\otimes_G A=\oplus_{n\in \Z} U_n$, where $U_n=A_n\otimes A_n$. If $\varphi\colon A \to B$ is a morphism in the category $\Z\!\alg_{\mathbb K}$, then ${\mathcal G}(\varphi)\colon A\otimes_G A \to B\otimes_G B$, with ${\mathcal G}(\varphi)(a_n\otimes a_n')=\varphi(a_n)\otimes\varphi(a_n')$.
 \begin{theorem}\label{preserve}
     The functor $\mathcal G$ preserves direct limits.
 \end{theorem}
 \begin{proof}
     Assume $(\{A_i\}_{i\in I}, \{e_{ji}\}_{i\le j})$ is a direct system in the category $\Z\!\alg_{\mathbb K}$ and  $\L A_i=A$. So, we have the following commutative diagrams for $j\ge i$:
\begin{center}
\begin{tikzcd}[cramped, column sep=small] 
A_i \arrow{r}{e_{ji}}  \arrow[rd,"e_{ki}"'] 
  & A_j \arrow{d}{e_{kj}} & A_i \arrow{r}{e_{ji}}\arrow[rd,"e_{i}"'] & A_j\arrow{d}{e_{j}}\\
    & A_k &  & A
\end{tikzcd}
\end{center}  
Applying the functor $\mathcal G$ we obtain the direct system $(\{ {\mathcal G}(A_i)\}_{i\in I}, {\mathcal G}(e_{ji})\}_{i\le j})$. Let us see that $\L {\mathcal G}(A_i)={\mathcal G}(A)$. If we have another $\Z$-graded associative $\mathbb K$-algebra $X$ and a family of morphism $\{f_i\}_{i \in I}$ in the category $\Z\!\alg_K$, such that the following diagrams are commutative:

\begin{center}
\begin{tikzcd}[cramped, column sep=normal] 
  \mathcal{G}(A_i) \arrow{r}{{\mathcal G}(e_{ji})}\arrow[rd,"f_{i}"'] & {\mathcal G}(A_j)\arrow{d}{f_{j}}\\
    & X
\end{tikzcd}
\end{center}
 We have for any $x_n\otimes y_n \in {\mathcal G}(A)$, with $x_n=e_i(a_n),\ y_n=e_j(b_n) \in A_n$ where $a_n\in (A_i)_n,\  b_n\in (A_j)_n$. There exists $k \ge i,  j \in I$ such that $x_n=e_ke_{ki}(a_n)$ and $y_n=e_ke_{kj}(b_n)$, with $u_n=e_{ki}(a_n)$ and  $v_n=e_{kj}(b_n) \in (A_k)_n$. So, we can define the morphism $\phi\colon  {\mathcal G}(A) \to X$ given by $\phi(x_n\otimes y_n):=f_k(u_n\otimes v_n)$. Let us see that it is well-defined. If $x_n\otimes y_n=e_q(u_n')\otimes e_q(v_n')$, then there exists $r\ge k, q$ with $e_re_{rk}(u_n)=e_k(u_n)=x_n=e_re_{rq}(u_n')=e_q(u_n')$ and by injectivity of $e_r$, then $e_{rk}(u_n)=e_{rq}(u_n')$. Analogously, $e_{rk}(v_n)=e_{rq}(v_n')$. Therefore, $f_k(u_n\otimes v_n)=f_r{\mathcal G}(e_{rk})(u_n\otimes v_n)=f_r(e_{rk}(u_n)\otimes e_{rk}(v_n))= f_r(e_{rq}(u_n')\otimes e_{rq}(v_n'))=f_r{\mathcal G}(e_{rq})(u_n'\otimes v_n')=f_q(u_n'\otimes v_n')$. So, we get ${\mathcal G}(\L A_i)\cong\L {\mathcal G}(A_i)$.
\end{proof} 
\begin{corollary}
    If $E=\L E_i$ and each quiver $E_i$ verifies $L_{\mathbb K}(\widehat{E_i})={\mathcal G}(L_{\mathbb K}(E_i))$, then $L_{\mathbb K}(\widehat{E})\cong {\mathcal G} (L_{\mathbb K}(E)) \cong L_{\mathbb K}(E)\otimes_G L_{\mathbb K}(E)$.
    
\end{corollary}
\begin{proof}
First, by Proposition \ref{sombrero}, we have $\widehat{E}=\L \widehat{E_i}$ implying $L_{\mathbb K}(\widehat{E})\cong L_{\mathbb K}(\L \widehat{E_i})$.
Using Corollary \ref{coro1} and Theorem \ref{preserve} we get $L_{\mathbb K}(\widehat{E})\cong\L L_{\mathbb K}(\widehat{E_i}) \cong \L {\mathcal G}(L_{\mathbb K}(E_i))\cong {\mathcal G}(\L (L_{\mathbb K}(E_i)))\cong {\mathcal G}(L_{\mathbb K}(E))\cong L_{\mathbb K}(E)\otimes_G L_{\mathbb K}(E)$.
\end{proof}

\begin{proposition}
If $E$ is a quiver whose connected components $\{E_i\}_{i=1}^{\infty}$ are finite or infinite lines quivers, then $L_{\mathbb K}(\widehat{E})\cong L_{\mathbb K}(E) \otimes_G L_{\mathbb K}(E)$.
\end{proposition}
\begin{proof}
\begin{comment}
    First, consider the direct system $(\{Q_i\}_{i=1}^{\infty}, \{f_{ji}\}_{i\le j})$, where $Q_i=\cup_{l=1}^iE_l$ and $f_{ji}$ are CK-monomorphism so that $\L Q_i = E$, then  
$\widehat{E}\cong \L \widehat{Q_i}$ by Proposition \ref{sombrero}.
    It is well known $L_{\mathbb K}(E)\cong \oplus_{i=1} ^{\infty} L_{\mathbb K}(E_i)$. Furthermore, by Propositions \ref{finiteline} and \ref{infiniteline}, we get $L_{\mathbb K}(\widehat{E_i})\cong L_{\mathbb K}(E_i)\otimes_G L_{\mathbb K}(E_i)$ for any $i$.  If we consider the direct sum as a direct limit of the direct system $(\{S_i\}_{i=1}^{\infty}, \{e_{ji}\}_{i\le j})$, where $S_i=\bigoplus_{l=1}^iL_{\mathbb K}(E_l)$ and the monomorphisms $e_{ji}\colon S_i \to S_j$ for $i\le j$, then $L_{\mathbb K}(E)\cong \oplus_{i=1} ^{\infty} L_{\mathbb K}(E_i)\cong \L S_i$. So, $L_{\mathbb K}(\widehat{E})\cong L_{\mathbb K}( \L \widehat{Q_i})\cong \L L_{\mathbb K}(\widehat{Q_i})$ by \cite[Lemma 2.5]{Goodearl}. On the other hand, by hypothesis  $L_{\mathbb K}(\widehat{E})\cong L_{\mathbb K}( \L \widehat{Q_i})\cong \L L_{\mathbb K}(\widehat{Q_i})\cong \L {\mathcal G}(L_{\mathbb K}(Q_i))\cong {\mathcal G}(L_{\mathbb K}(\L Q_i))\cong {\mathcal G}(L_{\mathbb K}(E))\cong L_{\mathbb K}(E)\otimes_G L_{\mathbb K}(E)$.
\end{comment}

We have $E=\sqcup_i E_i$ so that $\widehat{E}=\sqcup_{i,j}(E_i\times E_j)$. Then
$L_\K(\widehat{E})\cong \oplus_{i,j} L_\K(E_i\times E_j)$, and by Proposition  \ref{gradediso} we have $\oplus_{i,j} L_\K(E_i\times E_j)\cong \oplus_{i,j}\left[L_\K(E_i)\otimes_G L_\K(E_j)\right]\cong (\oplus_{i} L_\K(E_i))\otimes_G (\oplus_{j} L_\K(E_j))\cong L_\K(E)\otimes_G L_\K(E)$.

\end{proof}

\def\G{\mathcal{G}}
\begin{comment}
Assume that $A=\oplus_{n\in\Z}A_n$ and $B=\oplus_{n\in\Z}B_n$ are $\Z$-graded algebras.
Then we consider their point-wise direct sum $A\oplus B$ and compute $\G(A\oplus B)$.
%
\begin{lemma}
    $\G(A\oplus B)\cong\G(A)\oplus\G(B)\oplus (A\otimes_G B)\oplus (B\otimes_G A)$ where the four summands are ideals of the algebra
    $\G(A\oplus B)$.
\end{lemma}
\begin{proof}
Since $(A\oplus B)_n:=A_n\oplus B_n$ we have 
$\G(A\oplus B)=\oplus_{n\in\Z}[(A_n\oplus B_n)\otimes(A_n\oplus B_n)]=\oplus_n(A_n\otimes A_n\oplus B_n\otimes B_n\oplus A_n\otimes B_n\oplus B_n\otimes A_n)\cong\G(A)\oplus\G(B)\oplus(A\otimes_G B)\oplus(B\otimes_G A)$.
\end{proof}
%
Now, if $E=E_1\sqcup E_2$ and the connected components $E_i$ satisfy 
$L_K(\widehat{E_i})\cong\G(L_K(E_i))$ for $i=1,2$, then
$\widehat{E_1\sqcup E_2}=\widehat{E_1}\sqcup \widehat{E_2}\sqcup (E_1\times E_2)\sqcup (E_2\times E_1)$ so that 
$$L_K(\widehat{E_1\sqcup E_2})\cong L_K(\widehat{E_1})\oplus L_K(\widehat{E_2})\oplus L_K(E_1\times E_2)\oplus L_K(E_2\times E_1)\cong$$ $$\G(L_K(E_1))\oplus \G(L_K(E_2))\oplus L_K(E_1\times E_2)\oplus L_K(E_2\times E_1).$$
On the other hand, $\G(L_K(E_1\sqcup E_2))\cong\G(L_K(E_1)\oplus L_K(E_2))$, which applying the above lemma is isomorphic to 
$\G(L_K(E_1))\oplus \G(L_K(E_2))\oplus L_K(E_1)\otimes_G L_K(E_2)\oplus L_K(E_2)\otimes_G L_K(E_1)$.
\end{comment}
\color{black}
\section{Ring theoretic properties of Leavitt path algebra over Kronecker square}

\noindent In the view of the graph-theoretic properties shared by $Q$ and $\widehat{Q}$, it can easily be observed that the two Leavitt path algebras $L_{\mathbb K}(Q)$ and $L_{\mathbb K}(\widehat{Q})$ share many ring-theoretic properties.

\begin{theorem}
Let $Q$ be a quiver, and $\widehat{Q}$ its Kronecker square. Then

\begin{enumerate}
\item $L_{\mathbb K}(Q)$ is von Neumann regular if and only if $L_{\mathbb K}(\widehat{Q})$ is von Neumann regular (see \cite[Theorem 3.4.1]{AAS}). 
\item $L_{\mathbb K}(Q)$ is a left/right noetherian ring if and only if $L_{\mathbb K}(\widehat{Q})$ is left/right noetherian (see \cite[Corollary 4.2.14]{AAS}). 
\item $L_{\mathbb K}(Q)$ is a left/right artinian ring if and only if $L_{\mathbb K}(\widehat{Q})$ is left/right artinian (see \cite[Corollary 4.2.13]{AAS}). 
\end{enumerate}
\end{theorem}

\noindent But, there are some other ring-theoretic properties with respect to which $L_{\mathbb K}(Q)$ and $L_{\mathbb K}(\widehat{Q})$ show different behaviors. Recall that a subset $A \subseteq Q_0$ is said to be \textit{downward directed} if
for any $v_1, v_2 \in A$, there exists a vertex $v_3 \in A$ such that there is a path from $v_1$ to $v_3$ and a path from $v_2$ to $v_3$.
We say that a quiver $Q$ is downward directed when $Q_0$
is downward directed.

\begin{remark}
If $L_{\mathbb K}(Q)$ is a prime (primitive) ring, then $L_{\mathbb K}(\widehat{Q})$ need not be a prime (primitive) ring. This follows from the fact that if $Q$ is a downward directed quiver, then $\widehat Q$ need not be downward directed. Consider the following quiver $Q$:

 \[
\xymatrix{ && v_1\ar@{->}[rr] && v_2 && v_3 \ar@{->}[ll]} 
\]  

\noindent Clearly $Q$ is downward directed and hence, $L_{\mathbb K}(Q)$ is a prime ring. Now, the Kronecker square $\widehat{Q}$ is given as

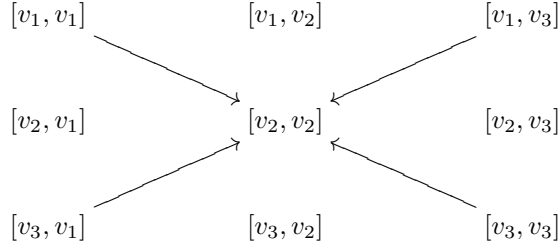
\[ 
\xymatrix{[v_1, v_1] \ar@{->}[drr] && [v_1, v_2]&&  [v_1, v_3]\ar@{->}[dll] \\
[v_2, v_1] && [v_2, v_2] &&[v_2, v_3] \\ [v_3, v_1] \ar@{->}[urr] && [v_3, v_2] &&[v_3, v_3] \ar@{->}[ull] }
\] 

\noindent For $[v_1, v_1]$ and $[v_1, v_2] \in \widehat Q_0 $, there is no vertex, say $[w_1,w_2]$, in $\widehat Q_0$ with a path from $[v_1,v_2]$ to $[w_1,w_2]$. Thus $\widehat Q$ is not downward directed and hence $L_{\mathbb K}(\widehat{Q})$ is not a prime ring.
\end{remark}

\begin{theorem} \label{prime}
If  $L_{\mathbb K}(\widehat{Q})$ is a prime ring, then $L_{\mathbb K}({Q}) $ is a prime ring. 
\end{theorem}

\begin{proof}
 Assume $L_{\mathbb K}(\widehat{Q})$ is a prime ring. Let $u,v\in Q_0$. Then $[u,u],[v,v] \in \widehat Q_0$. Since $L_{\mathbb K}(\widehat{Q})$ is a prime ring, $\widehat Q$ is downward directed. Hence, there exists $[w_1,w_2]\in \widehat Q_0$ such that 
 there is a path from $[u,u]$ to $[w_1,w_2] $ and a path from $ [v,v]$ to $[w_1,w_2]$. Thus we have $w_1 \in Q_0$ with a path from $u$ to $w_1 $ and a path from $v$ to $w_1$. This shows that $Q$ is downward directed and consequently, $L_{\mathbb K}({Q}) $ is a prime ring.
\end{proof} 

A quiver $Q$ is said to have the {\it countable separation property} if there exists a countable subset $S\subset Q_0$ such that for each $u\in Q_0$, there exists a vertex $v\in S$ with a path from $u$ to $v$.  

\begin{lemma}
A quiver $Q$ has the countable separation property if and only if $\widehat Q$ has the countable separation property. 
\end{lemma}

\begin{proof}
Assume $Q$ has the countable separation property. Then there exists a countable subset $S\subset Q_0$ such that for each $u\in Q_0$, there exists a $v\in S$ with a path from $u$ to $v$. Clearly $S\times S$ is a countable subset of $\widehat Q_0$. Take any vertex $[u, v]$ in $\widehat Q_0$. Now, there is a vertex $w_1\in S$ with a path from $u$ to $w_1$ and a vertex $w_2\in S$ with a path from $v$ to $w_2$. Clearly, then $[w_1, w_2]\in S\times S$ with a path from $[u, v]$ to $[w_1, w_2]$. The converse follows easily. 
\end{proof}

\noindent It is known that a Leavitt path algebra $L_{\mathbb K}({Q})$ is a right (left) primitive ring if and only if $Q$ is downward directed, $Q$ satisfies the Condition (L) and $Q$ has the countable separation property. Thus, from the above lemma and the proof of Theorem \ref{prime}, we have the following.

\begin{theorem}
If  $L_{\mathbb K}(\widehat{Q})$ is a right (left) primitive ring, then $L_{\mathbb K}({Q}) $ is a right (left) primitive ring. 
\end{theorem}  

Now, let us recall the basics of the notion of the Gelfand-Kirillov dimension of an algebra. Let $A$ be an algebra which is generated by a finite-dimensional subspace $V$. Let $V^n$ denote the span of all products $v_1v_2\cdots v_k$, $v_i\in V$, $k\le n$. Then $V=V^1\subseteq V^2\subseteq \ldots$, $A=\cup_{n\ge 1}V^n$, and $\dim(V^n)<\infty$. If $W$ is another finite-dimensional subspace that generates $A$, then $$\lim_{n\rightarrow \infty} \frac{\text{log}(\dim(V^n))}{\text{log}(n)}=\lim_{n\rightarrow \infty} \frac{\text{log}(\dim(W^n))}{\text{log}(n)}.$$ If $\dim(V^n)$ is polynomially bounded, then the Gelfand-Kirillov dimension of $A$ is defined as $\GKdim(A)=\lim\;\text{sup}_{n\rightarrow \infty} \frac{\text{log}(\dim(V^n))}{\text{log}(n)}$. The Gelfand-Kirillov dimension of $A$ does not depend on the choice of generating space $V$ as long as $V$ is finite-dimensional.  

\begin{lemma}
$L_{\mathbb K}(Q)$ has finite polynomially bounded growth if and only if $L_{\mathbb K}(\widehat{Q})$ has polynomially bounded growth. 
\end{lemma}  

\begin{proof}
It is shown in \cite{CW} that any two distinct cycles in $Q$ do not have a common vertex if and only if any two cycles in $\widehat{Q}$ do not have a common vertex. It is shown in \cite{Z1} that for a finite quiver $Q$, the Leavitt path algebra $L_K(Q)$ has polynomially bounded growth if and only if no two distinct cycles in $Q$ have a common vertex. Thus it follows that $L_{\mathbb K}(Q)$ has finite polynomially bounded growth if and only if $L_{\mathbb K}(\widehat{Q})$ has polynomially bounded growth.   
\end{proof}

 Let $Q$ be a quiver where any two distinct cycles are disjoint. We know that if $d_1$ is the maximal length of a chain of cycles in $Q$ and $d_2$ is the maximal length of chain of cycles with an exit, then $\GKdim(L_{\mathbb K}(Q))=\max(2d_1-1, 2d_2)$ \cite{Z1}. Therefore, to see the relation between the GK-dimension of Leavitt path algebra over a quiver $Q$ and its Kronecker square $\widehat Q$, we first consider the relation between the maximal length of chain of cycles in $Q$ and $\widehat Q$. 
 
Recall that for two cycles $C',C''$ of $Q$ we write $C'\Rightarrow C''$ if there is a path in $Q$ starting on $C'$ and ending on $C''$. Under the  assumption that any two distinct cycles are disjoint, $\Rightarrow$ is a partial order on the set of cycles of $Q$.

%\begin{definition}
\noindent A \emph{chain of cycles} of length $k$ in a quiver is a sequence
$C_{1},\dots,C_{k}$ of \emph{distinct} cycles with
$C_{1}\Rightarrow C_{2}\Rightarrow\cdots\Rightarrow C_{k}$.
The chain is said to have \emph{an exit} if its last cycle $C_{k}$
has an exit. 
%\end{definition}

Note that if $D=[d_1, d_2]$ is a cycle of $\widehat{Q}$ then both $d_1, d_2$ are closed paths in $Q$. Since we have assumed that distinct cycles of $Q$ are disjoint, we have that every closed path through a vertex $u$ is a
power of the unique cycle of $Q$ through $u$. So, we can write $D=[c^n, d^m]$ for some cycles $c$ and $d$ in $Q$.  
\begin{lemma}
     Let $d_1,d_1^\prime$ be the maximal length of a chain of cycles in $Q$,$\widehat Q$, respectively, and  let $d_2,d_2^\prime$ be the maximal length of a chain of cycles with an exit in $Q$,$\widehat Q$, respectively. If $d_1 = k$ for a positive integer $k$, then 
    \begin{enumerate}
        \item $d_1^ \prime = 2k-1$,
        \item if $d_2 < d_1$ , then $d_2 ^ \prime = 2k-2$, otherwise $d_2 ^ \prime = 2k-1$.
    \end{enumerate}
\end{lemma}

\begin{proof} 
\begin{enumerate}
 \item It is already shown in \cite{CW} but we give here an argument based on induction for the sake of completeness. If $d_1=1$, then clearly, $d_1 ^ \prime = 1$. Now, assume that if $d_1 = r$ then $d_1 ^ \prime = 2r-1 $. Let $d_1= r+1$, so there exists a chain of cycles $c_1 \Rightarrow c_2 \Rightarrow \cdots \Rightarrow c_r \Rightarrow c_{r+1}$.
    Let $[c_1,c_1] \Rightarrow [c_1,c_2] \Rightarrow \cdots \Rightarrow [c_r,c_r]$ be a chain of cycles in $\widehat Q$ with length $d_1 ^ \prime = 2r-1 $. As $c_r \Rightarrow c_{r+1}$, then we can extend our chain in $\widehat Q $ to be $[c_1,c_1] \Rightarrow [c_1,c_2] \Rightarrow \cdots \Rightarrow [c_r,c_r] \Rightarrow [c_r,c_{r+1}] \Rightarrow [c_{r+1},c_{r+1}]$ or  $[c_1,c_1] \Rightarrow [c_2,c_2] \Rightarrow \cdots \Rightarrow [c_r,c_r] \Rightarrow [c_{r+1},c_r] \Rightarrow [c_{r+1},c_{r+1}]$. In both cases we have a chain of cycles in $\widehat{Q}$ of length $(2r-1)+2=2r+1$. This shows $d_1 ^ \prime \ge 2r+1$. 
    Let $D_{1}\Rightarrow\cdots\Rightarrow D_{d_1^\prime}$ be a chain of
distinct cycles in $\widehat{Q}$. Let $D_j=[c_{j},d_{j}]$, where $c_{j},d_{j}$ are cycles in $Q$. We
have a chain 
\[
    [c_{1}, d_{1}]\Rightarrow \cdots\Rightarrow
    [c_{d_1^\prime}, d_{d_1^\prime}]
\]
and consecutive pairs are distinct. Hence the pairs are strictly increasing. Each
coordinate yields a (weakly increasing) chain in the cycle poset
of $Q$. After collapsing equal consecutive values we obtain
strict chains of lengths $r_{1},r_{2}\le r+1$, and
\[
    d_1^\prime-1\le(r_{1}-1)+(r_{2}-1)\le 2r.
\]
Therefore $d_1^\prime\le 2r+1$. Thus we have $d_1^ \prime = 2k-1$. 
\item 
Suppose $d_{2}<d_{1}$, i.e.\ $d_{2}=k-1$.
By assumption, in every length-$k$ chain
$c_{1}\Rightarrow\cdots\Rightarrow c_{k}$ in $Q$ the last cycle
$c_{k}$ has no exit (otherwise we would have $d_{2}=k$).

Fix a chain $c_{1}\Rightarrow\cdots\Rightarrow c_{k}$ of length
$k$ in $Q$. This gives rise to the following chain of cycles in $\widehat{Q}$
\[
    [c_1, c_1]\Rightarrow [c_1, c_2] \Rightarrow \dots, [c_1, c_{k-1}]\Rightarrow  [c_2, c_{k-1}] \Rightarrow [c_3, c_{k-1}], \dots [c_k, c_{k-1}]
\]
which has $2k-2$ entries. Since $c_{k-1}$ has an
exit, $[c_k, c_{k-1}]$ has an exit. Hence $d_{2}'\ge 2k-2$.

Next, we claim $d_{2}'\le 2k-2$. Suppose to the contrary that
$D_{1}\Rightarrow\cdots\Rightarrow D_{2k-1}$ is a chain of length
$2k-1$ in $\widehat{Q}$ with $D_{2k-1}$ having an exit. By the proof of the upper bound in previous part, the projected pairs
form a strictly increasing chain of length $2k-1$ in the product
order, with both coordinate-projection chains of length exactly
$k$ in $Q$. Write the coordinate-1 chain as
$a_{1}\Rightarrow a_{2}\Rightarrow \cdots\Rightarrow a_{k}$ and the coordinate-2
chain as $b_{1}\Rightarrow b_{2}\Rightarrow \cdots\Rightarrow b_{k}$ (both chains in
$Q$). The last projected pair is $[a_{k},b_{k}]$.

Now since $d_{2}<d_{1}=k$, every length-$k$ chain in $Q$ ends in
a cycle without an exit. So $a_{k}$ has no exit in $Q$, and
$b_{k}$ has no exit in $Q$. Clearly then $D_{2k-1}$ has no exit in $\widehat{Q}$, contradicting the
hypothesis. Therefore $d_{2}'\le 2k-2$.

Combining the bounds, $d_{2}'=2k-2$. 
    \end{enumerate}
    \end{proof}
If $Q$ is acyclic, then $\widehat{Q}$ is acyclic as well and in that case, $\GKdim (L_{\mathbb K}(Q))=0=\GKdim (L_{\mathbb K}(\widehat Q))$. 
\begin{theorem}
For a finite quiver $Q$, if $\GKdim (L_{\mathbb K}(Q))=n$, with $n\neq 0$, then 
\begin{equation*}
\GKdim(L_{\mathbb K}(\widehat Q))=  \begin{cases}
     2n-1, & \quad n \hspace{.2cm} \text{is odd}\\
     2n-2, & \quad n \hspace{.2cm} \text{is even}    
\end{cases}
\end{equation*}
\end{theorem}

\begin{proof} 
Let $\GKdim(L_{\mathbb K}(Q))=n$ where $n$ is odd. In this case, $2 d_2 < 2 d_1 -1 $ and $n = 2 d_1 -1 $. Now $ d_2 ^ \prime <  d_1 ^ \prime $ so we have  $2 d_2 ^ \prime < 2 d_1 ^ \prime -1$. Therefore, $\GKdim(L_{\mathbb K}(\widehat Q)) = \max ( 2 d_1 ^ \prime -1, 2 d_2 ^ \prime) = 2 d_1 ^ \prime -1 = 2(2d_1-1)=2n-1$.
Now, let $\GKdim(L_{\mathbb K}(Q))=n$ where $n$ is even. Then, $d_1 = d_2$ and $n = 2 d_2$. Also, $d_1 ^ \prime = d_2 ^ \prime = 2 d_1 -1$. Hence, $\GKdim(L_{\mathbb K}(\widehat Q)) = \max ( 2 d_1 ^ \prime -1, 2 d_2 ^ \prime)= 2 d_2 ^ \prime = 2 (2 d_1 -1)=2 (n-1)= 2n-2$.
\end{proof}

\section{Isomorphisms between Leavitt path algebras}

\noindent The following question is quite natural. 

\begin{question} \label{Question1}
 Suppose $Q_1$ and $Q_2$ are two quivers with $L_{\mathbb K}(Q_1) \cong L_{\mathbb K}(Q_2)$. Does it imply $L_{\mathbb K}(\widehat {Q_1}) \cong L_{\mathbb K}(\widehat{Q_2})$?
\end{question}

We answer this question in the negative in the example below. 

\begin{example} \rm
Let  $Q_1$ and $Q_2$ be quivers given below: 
     \vspace{.3cm}
     % https://q.uiver.app/#q=WzAsOCxbMSwwLCJ2XzEiXSxbMCwwLCJRXzE6Il0sWzIsMCwidl8yIl0sWzMsMCwidl8zIl0sWzAsMSwiUV8yOiJdLFsxLDEsInZfMSJdLFsyLDEsInZfMiJdLFszLDEsInZfMyJdLFswLDJdLFsyLDNdLFs1LDZdLFs3LDZdXQ==
\[\begin{tikzcd}
	{Q_1:} & {v_1} & {v_2} & {v_3} \\
	{Q_2:} & {v_1} & {v_2} & {v_3}
	\arrow[from=1-2, to=1-3]
	\arrow[from=1-3, to=1-4]
	\arrow[from=2-2, to=2-3]
	\arrow[from=2-4, to=2-3]
\end{tikzcd}\]

Clearly their Leavitt path algebras are isomorphic, as 
$$L_{\mathbb K}(Q_1)\cong  M_3 (\mathbb K) \cong L_{\mathbb K}(Q_2).$$

Now, the Kronecker squares are as given below. \\

% https://q.uiver.app/#q=WzAsMjAsWzEsMCwiW3ZfMSx2XzFdIl0sWzIsMCwiW3ZfMSx2XzJdIl0sWzMsMCwiW3ZfMSx2XzNdIl0sWzEsMSwiW3ZfMix2XzFdIl0sWzIsMSwiW3ZfMix2XzJdIl0sWzMsMSwiW3ZfMix2XzNdIl0sWzEsMiwiW3ZfMyx2XzFdIl0sWzIsMiwiW3ZfMyx2XzJdIl0sWzMsMiwiW3ZfMyx2XzNdIl0sWzAsMSwiXFx3aWRlaGF0e1F9XzE6Il0sWzAsNSwiXFx3aWRlaGF0e1F9XzI6Il0sWzEsNCwiW3ZfMSx2XzFdIl0sWzIsNCwiW3ZfMSx2XzJdIl0sWzMsNCwiW3ZfMSx2XzNdIl0sWzEsNSwiW3ZfMix2XzFdIl0sWzIsNSwiW3ZfMix2XzJdIl0sWzMsNSwiW3ZfMix2XzNdIl0sWzEsNiwiW3ZfMyx2XzFdIl0sWzIsNiwiW3ZfMyx2XzJdIl0sWzMsNiwiW3ZfMyx2XzNdIl0sWzAsNF0sWzEsNV0sWzQsOF0sWzMsN10sWzExLDE1XSxbMTMsMTVdLFsxOSwxNV0sWzE3LDE1XV0=
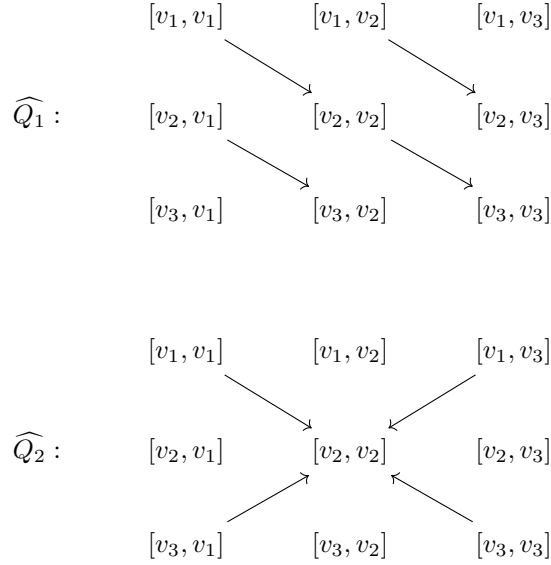
\[\begin{tikzcd}
	& {[v_1,v_1]} & {[v_1,v_2]} & {[v_1,v_3]} \\
	{\widehat{Q_1}:} & {[v_2,v_1]} & {[v_2,v_2]} & {[v_2,v_3]} \\
	& {[v_3,v_1]} & {[v_3,v_2]} & {[v_3,v_3]} \\
	\\
	& {[v_1,v_1]} & {[v_1,v_2]} & {[v_1,v_3]} \\
	{\widehat{Q_2}:} & {[v_2,v_1]} & {[v_2,v_2]} & {[v_2,v_3]} \\
	& {[v_3,v_1]} & {[v_3,v_2]} & {[v_3,v_3]}
	\arrow[from=1-2, to=2-3]
	\arrow[from=1-3, to=2-4]
	\arrow[from=2-2, to=3-3]
	\arrow[from=2-3, to=3-4]
	\arrow[from=5-2, to=6-3]
	\arrow[from=5-4, to=6-3]
	\arrow[from=7-2, to=6-3]
	\arrow[from=7-4, to=6-3]
\end{tikzcd}\]

 \noindent The Leavitt path algebras of $\widehat {Q_1}$ and $\widehat{Q_2}$ are given as 
$L_{\mathbb K}(\widehat{ Q_1})\cong M_3(\mathbb K) \oplus M_2^2 (\mathbb K) \oplus  \mathbb K^{2}$  %$L_{\mathbb K}(\widehat{E_2})\cong M_2(\mathbb K) \oplus M_2 (\mathbb K) \oplus M_2 (\mathbb K) \oplus M_2 (\mathbb K) \oplus \mathbb K^{4}$.
\vspace{.2cm} \\
$L_{\mathbb K}(\widehat{Q_2})\cong M_5(\mathbb K) \oplus \mathbb K^{4}$. 

Clearly, $L_{\mathbb K}(\widehat{ Q_1})$ and $L_{\mathbb K}(\widehat{ Q_2})$ are not isomorphic. This answers Question \ref{Question1} in the negative. 
\end{example}

Next, we ask this question when there is a graded isomorphism between Leavitt path algebras over Kronecker squares over two quivers. 

\begin{question} \label{Q2}  
 Suppose $Q_1$ and $Q_2$ are two quivers with $L_{\mathbb K}(Q_1) \cong_{gr} L_{\mathbb K}(Q_2)$. Does it imply that $L_{\mathbb K}(\widehat {Q_1}) \cong_{gr} L_{\mathbb K}(\widehat{Q_2})$?
\end{question}

To answer this, let us mention the following: 

\begin{theorem} \cite{Hazrat1} \label{graded isom.}
    Let $Q$ be a finite acyclic quiver with sinks \(\{v_1, \ldots, v_t\}\). For any sink \(v_s\), let \(\{p_{v_s}^i \mid 1 \leq i \leq n(v_s)\}\) be the set of all paths which end in \(v_s\). Then there is a $\mathbb Z$-graded isomorphism
    \[ L_{\mathbb K}(Q) \cong_{\text{gr}} \bigoplus_{s=1}^t M_{n(v_s)}(\mathbb K)(|p_1^{v_s}|, \cdots, |p_{n(v_s})^{v_s}|).\]
    Let $Q'$ be another acyclic quiver with sinks \(\{u_1, \ldots, u_k\}\), and let \(\{p_{u_s}^i \mid 1 \leq i \leq n(u_s)\}\) be the set of all paths which end in \(u_s\). Then \(L_{\mathbb K}(Q) \cong_{\text{gr}} L_{\mathbb K}(Q')\) if and only if \(k = t\), and after a permutation of indices, \(n(v_s) = n(u_s)\) and \(\{|p_{v_s}^i| \mid 1 \leq i \leq n(v_s)\}\) and \(\{|p_{u_s}^i| \mid 1 \leq i \leq n(u_s)\}\) present the same list.
\end{theorem}

\begin{example}\rm
(See \cite{Hazrat1}, Example 4.14) Let $Q_1$ and $Q_2$ be the following:
% https://q.uiver.app/#q=WzAsMTIsWzAsMCwiUV8xOiJdLFsxLDAsInZfMSJdLFsyLDAsInZfMiJdLFszLDAsInZfMyJdLFs0LDAsInZfNCJdLFs1LDAsInZfNSJdLFswLDMsIlFfMjoiXSxbMiwzLCJ2XzEiXSxbMywzLCJ2XzIiXSxbNCwzLCJ2XzMiXSxbNCwyLCJ2XzUiXSxbMyw0LCJ2XzQiXSxbMSwyXSxbMiwzXSxbNCwzXSxbNSw0XSxbNyw4XSxbOCw5XSxbMTAsOV0sWzExLDhdXQ==
\[\begin{tikzcd}
	{Q_1:} & {v_1} & {v_2} & {v_3} & {v_4} & {v_5} \\
	\\
	&&&& {v_5} \\
	{Q_2:} && {v_1} & {v_2} & {v_3} \\
	&&& {v_4}
	\arrow[from=1-2, to=1-3]
	\arrow[from=1-3, to=1-4]
	\arrow[from=1-5, to=1-4]
	\arrow[from=1-6, to=1-5]
	\arrow[from=3-5, to=4-5]
	\arrow[from=4-3, to=4-4]
	\arrow[from=4-4, to=4-5]
	\arrow[from=5-4, to=4-4]
\end{tikzcd}\]

\noindent By the theorem above, both $L_{\mathbb K}(Q_1)$ and $L_{\mathbb K}(Q_2)$ are graded isomorphic to $M_5(\mathbb K)(0, 1, 1, 2, 2)$ and thus we have \(L_{\mathbb K}(Q_1) \cong_{\text{gr}} L_{\mathbb K}(Q_2)\). 

Now, let us see the Kronecker square of these two quivers: 

% https://q.uiver.app/#q=WzAsMjYsWzIsMCwiW3ZfMSx2XzFdIl0sWzQsMCwiW3ZfMSx2XzJdIl0sWzYsMCwiW3ZfMSx2XzNdIl0sWzgsMCwiW3ZfMSx2XzRdIl0sWzEwLDAsIlt2XzEsdl81XSJdLFsyLDIsIlt2XzIsdl8xXSJdLFs0LDIsIlt2XzIsdl8yXSJdLFs2LDIsIlt2XzIsdl8zXSJdLFs4LDIsIlt2XzIsdl80XSJdLFsxMCwyLCJbdl8yLHZfNV0iXSxbMiw0LCJbdl8zLHZfMV0iXSxbNCw0LCJbdl8zLHZfMl0iXSxbNiw0LCJbdl8zLHZfM10iXSxbOCw0LCJbdl8zLHZfNF0iXSxbMTAsNCwiW3ZfMyx2XzVdIl0sWzIsNiwiW3ZfNCx2XzFdIl0sWzQsNiwiW3ZfNCx2XzJdIl0sWzYsNiwiW3ZfNCx2XzNdIl0sWzgsNiwiW3ZfNCx2XzRdIl0sWzEwLDYsIlt2XzQsdl81XSJdLFsxMCw4LCJbdl81LHZfNV0iXSxbOCw4LCJbdl80LHZfNV0iXSxbNiw4LCJbdl8zLHZfNV0iXSxbNCw4LCJbdl8yLHZfNV0iXSxbMiw4LCJbdl8xLHZfNV0iXSxbMCw0LCJcXHdpZGVoYXR7UV8xfToiXSxbMCw2XSxbNiwxMl0sWzE4LDEyXSxbMjAsMThdLFszLDddLFsxLDddLFsyMywxN10sWzIxLDE3XSxbNCw4XSxbOCwxMl0sWzksMTNdLFs1LDExXSxbMTUsMTFdLFsyNCwxNl0sWzE2LDEyXSxbMTksMTNdXQ==
\[\begin{tikzcd}  [column sep=small, row sep=small, nodes={inner sep=2pt, scale=1.3}]
	&& {[v_1,v_1]} && {[v_1,v_2]} && {[v_1,v_3]} && {[v_1,v_4]} && {[v_1,v_5]} \\
	\\
	&& {[v_2,v_1]} && {[v_2,v_2]} && {[v_2,v_3]} && {[v_2,v_4]} && {[v_2,v_5]} \\
	\\
	{\widehat{Q_1}:} && {[v_3,v_1]} && {[v_3,v_2]} && {[v_3,v_3]} && {[v_3,v_4]} && {[v_3,v_5]} \\
	\\
	&& {[v_4,v_1]} && {[v_4,v_2]} && {[v_4,v_3]} && {[v_4,v_4]} && {[v_4,v_5]} \\
	\\
	&& {[v_5,v_1]} && {[v_5,v_2]} && {[v_5,v_3]} && {[v_5,v_4]} && {[v_5,v_5]}
	\arrow[from=1-3, to=3-5]
	\arrow[from=1-5, to=3-7]
	\arrow[from=1-9, to=3-7]
	\arrow[from=1-11, to=3-9]
	\arrow[from=3-3, to=5-5]
	\arrow[from=3-5, to=5-7]
	\arrow[from=3-9, to=5-7]
	\arrow[from=3-11, to=5-9]
	\arrow[from=7-3, to=5-5]
	\arrow[from=7-5, to=5-7]
	\arrow[from=7-9, to=5-7]
	\arrow[from=7-11, to=5-9]
	\arrow[from=9-3, to=7-5]
	\arrow[from=9-5, to=7-7]
	\arrow[from=9-9, to=7-7]
	\arrow[from=9-11, to=7-9]
\end{tikzcd}\]
\vspace{1cm}
% https://q.uiver.app/#q=WzAsMjYsWzIsMCwiW3ZfMSx2XzFdIl0sWzQsMCwiW3ZfMSx2XzJdIl0sWzYsMCwiW3ZfMSx2XzNdIl0sWzgsMCwiW3ZfMSx2XzRdIl0sWzEwLDAsIlt2XzEsdl81XSJdLFsyLDIsIlt2XzIsdl8xXSJdLFs0LDIsIlt2XzIsdl8yXSJdLFs2LDIsIlt2XzIsdl8zXSJdLFs4LDIsIlt2XzIsdl80XSJdLFsxMCwyLCJbdl8yLHZfNV0iXSxbMiw0LCJbdl8zLHZfMV0iXSxbNCw0LCJbdl8zLHZfMl0iXSxbNiw0LCJbdl8zLHZfM10iXSxbOCw0LCJbdl8zLHZfNF0iXSxbMTAsNCwiW3ZfMyx2XzVdIl0sWzIsNiwiW3ZfNCx2XzFdIl0sWzQsNiwiW3ZfNCx2XzJdIl0sWzYsNiwiW3ZfNCx2XzNdIl0sWzgsNiwiW3ZfNCx2XzRdIl0sWzEwLDYsIlt2XzQsdl81XSJdLFsxMCw4LCJbdl81LHZfNV0iXSxbOCw4LCJbdl81LHZfNF0iXSxbNiw4LCJbdl81LHZfM10iXSxbNCw4LCJbdl81LHZfMl0iXSxbMiw4LCJbdl81LHZfMV0iXSxbMCw0LCJcXHdpZGVoYXR7UV8yfToiXSxbMCw2XSxbNiwxMl0sWzEsN10sWzUsMTFdLFsyMCwxMiwiIiwwLHsiY3VydmUiOjJ9XSxbMyw2XSxbNCw3XSxbOCwxMV0sWzksMTJdLFsxNSw2XSxbMTYsN10sWzE4LDYsIiIsMix7ImN1cnZlIjotMn1dLFsxOSw3LCIiLDIseyJjdXJ2ZSI6Mn1dLFsyNCwxMV0sWzIzLDEyXSxbMjEsMTEsIiIsMCx7ImN1cnZlIjotMn1dXQ==
\[\begin{tikzcd} [column sep=small, row sep=small, nodes={inner sep=2pt, scale=1.3}]
	&& {[v_1,v_1]} && {[v_1,v_2]} && {[v_1,v_3]} && {[v_1,v_4]} && {[v_1,v_5]} \\
	\\
	&& {[v_2,v_1]} && {[v_2,v_2]} && {[v_2,v_3]} && {[v_2,v_4]} && {[v_2,v_5]} \\
	\\
	{\widehat{Q_2}:} && {[v_3,v_1]} && {[v_3,v_2]} && {[v_3,v_3]} && {[v_3,v_4]} && {[v_3,v_5]} \\
	\\
	&& {[v_4,v_1]} && {[v_4,v_2]} && {[v_4,v_3]} && {[v_4,v_4]} && {[v_4,v_5]} \\
	\\
	&& {[v_5,v_1]} && {[v_5,v_2]} && {[v_5,v_3]} && {[v_5,v_4]} && {[v_5,v_5]}
	\arrow[from=1-3, to=3-5]
	\arrow[from=1-5, to=3-7]
	\arrow[from=1-9, to=3-5]
	\arrow[from=1-11, to=3-7]
	\arrow[from=3-3, to=5-5]
	\arrow[from=3-5, to=5-7]
	\arrow[from=3-9, to=5-5]
	\arrow[from=3-11, to=5-7]
	\arrow[from=7-3, to=3-5]
	\arrow[from=7-5, to=3-7]
	\arrow[bend right=20, from=7-9, to=3-5]
	\arrow[bend right=20, from=7-11, to=3-7]
	\arrow[from=9-3, to=5-5]
	\arrow[from=9-5, to=5-7]
	\arrow[bend left=20, from=9-9, to=5-5]
	\arrow[bend left=20, from=9-11, to=5-7]
\end{tikzcd}\]

By Theorem \ref{graded isom.}, 
\[
    L_{K}(\widehat{Q_1})\;\cong_{gr}\;
    M_9(K)(0,1,1,1,1,2,2,2,2)\oplus
    M_3(K)(0,1,1)^{\oplus 4}\oplus K^{\oplus 4}
\]
and
\[
    L_{K}(\widehat{Q_2})\;\cong_{gr}\;
    M_9(K)(0,1,1,1,1,2,2,2,2)\oplus
    M_5(K)(0,1,1,2,2)^{\oplus 4}\oplus
    M_3(K)(0,1,1)^{\oplus 2}\oplus K^{\oplus 2}.
\]
The block-size multisets $\{9,3,3,3,3,1,1,1,1\}$ and
$\{9,5,5,5,5,3,3,1,1\}$ differ, so
\[
    L_{K}(\widehat{Q_1})\;\not\cong_{gr}\;L_{K}(\widehat{Q_2}).
\]
Thus we have answered Question \ref{Q2} in the negative. 
\end{example}

A natural question to look at is the following.
\begin{question} \label{Q}
    Let $Q$ and $Q'$ be two finite quivers with $L_{\mathbb K}(\widehat{Q})\cong L_{\mathbb K}(\widehat{Q'})$. Does this imply $L_{\mathbb K}(Q)\cong L_{\mathbb K}(Q')$?
\end{question}

We answer this question in the positive for the case when the quivers $Q$ and $Q'$ are finite acyclic.  The structure of finite-dimensional Leavitt path algebras is known and it will be used in answering the above question.

\begin{theorem}[\cite{locallyfinite}]\label{prop:struct-acyclic}
   Let $Q$ be a finite acyclic quiver with sinks
$w_{1},\dots,w_{s}$. Then
\[
    L_{\mathbb K}(Q)\;\cong\;\bigoplus_{i=1}^{s} M_{n_{i}^{Q}}(\mathbb K),
\]
where $n_{i}^{Q}$ is the number of paths in $Q$ ending at the sink
$w_{i}$ (paths of length $0$ included). Consequently $L_{\mathbb K}(Q)$ is
finite-dimensional and semisimple, and its $\mathbb K$-algebra isomorphism
class is determined by the multiset
\[
    \mathcal N(Q)=\{n_{1}^{Q},\dots,n_{s}^{Q}\}.
\]

\end{theorem}

The following observations will prove helpful. 

\begin{proposition}\label{proposition:isopoints}
   Assume that $Q$ has no isolated vertices. Then a vertex $[u,v]\in \widehat{Q}_0$ is isolated if and only if $u$ is a sink (source) and $v$ is a source (sink). If we let $\text{Sinks}_Q$ denote the number of sinks in $Q$ and $\text{Sources}_Q$ denote the number of sources in $Q$, then $2(\text{{Sinks}}_Q\cdot\text{{Sources}}_Q)$ gives the number of isolated vertices in $\widehat{Q}$.
\end{proposition}
\begin{proof}
    In order for $[u,v]$ to be an isolated vertex, we cannot have an arrow going in or out of the vertex. No arrows terminate at $[u,v]$ if and only if $u$ or $v$ is a source, likewise, no arrows begin at $[u,v]$ if and only if $u$ or $v$ is a sink. Combining these observations, $[u,v]$ is isolated if and only if one of the components is a sink and the other is a source.

    The equation yielding the number of isolated vertices follows from the fact that there are $\text{Sinks}_Q\cdot \text{Sources}_Q$ vertices of the form $[u,v]$, where $u$ is a sink and $v$ is a source. Switching the order in which the sink and the source occurs also yields an isolated vertex, hence the number of isolated vertices in $\widehat{Q}$ is $2(\text{Sinks}_Q\cdot \text{Sources}_Q)$.
\end{proof}

In conjunction with Theorem \ref{prop:struct-acyclic}, the previous proposition ensures ${\mathbb K}^{2(\text{Sinks}_Q\cdot \text{Sources}_Q)}$ is a direct summand of $L_{\mathbb K}(\widehat{Q})$. Noe that the Theorem \ref{prop:struct-acyclic} indicates that the number of sinks of $\widehat{Q}$ will help determine the number of direct summands in the decomposition of $L_{\mathbb K}(\widehat{Q})$. By the definition of $\widehat{Q}$, we have that $[u,v]$ is a sink of $\widehat{Q}$ if and only if $u$ or $v$ is a sink. As such, we have the equation

\begin{align*}\text{Sinks}_{\widehat{Q}}&=\text{Sinks}_Q^2+2\text{Sinks}_Q\cdot (|Q_0|-\text{Sinks}_Q)\\ &=2\text{Sinks}_Q\cdot |Q_0|-\text{Sinks}_Q^2\\ &=\text{Sinks}_Q(2|Q_0|-\text{Sinks}_Q) \end{align*}

Observe that we count isolated vertices as sinks in $\widehat{Q}$. The number of vertices that are not isolated is given by \begin{equation}\label{eqn:noniso}\text{Sinks}_Q^2+2\text{Sinks}_Q\cdot (|Q_0|-\text{Sinks}_Q-\text{Sources}_Q).\end{equation}

Using the equation for $\text{Sinks}_{\widehat{Q}}$ above yields the following.

\begin{proposition}\label{proposition:sinks}Suppose $Q$ and $Q'$ are finite acyclic quivers. If $L_{\mathbb K}(\widehat{Q})\cong L_{\mathbb K}(\widehat{Q'})$ then $\text{Sinks}_Q=\text{Sinks}_{Q'}$.\end{proposition}

\begin{proof}  We know that if $Q$ and $Q'$ are finite acyclic quivers, then so are $\widehat{Q}$ and $\widehat{Q'}$. By Theorem \ref{prop:struct-acyclic}, we have $L_{\mathbb K}(\widehat Q) \cong (\bigoplus_{j=1}^{l}\mathbb{M}_{n_j}({\mathbb K}))$ and $L_{\mathbb K}(\widehat Q') \cong (\bigoplus_{j=1}^{b}\mathbb{M}_{n'_j}({\mathbb K}))$. Then the Goldie dimension of $L_{\mathbb K}(\widehat Q)$ as a right $L_{\mathbb K}(\widehat Q)$-module is $\sum_{j} n_{j}=|Q^2_0|$. Similarly, the Goldie dimension of $L_{\mathbb K}(\widehat Q')$ as a right $L_{\mathbb K}(\widehat Q)$-module is $\sum_{j} n'_{j}=|Q'^2_0|$. Since $L_{\mathbb K}(\widehat{Q})\cong L_{\mathbb K}(\widehat{Q'})$, they have the same Goldie dimension and so, $|Q^2_0|=|Q'^2_0|$. This yields $|Q_0|=|Q'_0|$. 

We must have that the number of direct summands of the form $M_{n_j}({\mathbb K})$ in both $L_{\mathbb K}(\widehat{Q})$ and $L_{\mathbb K}(\widehat{Q'})$ must be equal, as such, $\widehat{Q}$ and $\widehat{Q'}$ must have the same number of sinks, by Theorem \ref{prop:struct-acyclic}. Additionally, as $|Q_0|=|Q'_0|$, $\text{Sinks}_Q\cdot \text{Sources}_Q=\text{Sinks}_{Q'}\cdot \text{Sources}_{Q'}$ by Proposition \ref{proposition:isopoints}. The proof then follows from applying algebra to (\ref{eqn:noniso}):\begin{align*}&\text{Sinks}_Q^2+2\text{Sinks}_Q\cdot (|Q_0|-\text{Sinks}_Q-\text{Sources}_Q)=\text{Sinks}_{Q'}^2+2\text{Sinks}_{Q'}\cdot (|Q_0|-\text{Sinks}_{Q'}-\text{Sources}_{Q'})\\ &\Rightarrow 2\text{Sinks}_Q\cdot|Q_0|-\text{Sinks}_Q^2-2\text{Sinks}_Q\cdot \text{Sources}_Q=2\text{Sinks}_{Q'}\cdot|Q_0|-\text{Sinks}_{Q'}^2-2\text{Sinks}_{Q'}\cdot \text{Sources}_{Q'}\\ &\Rightarrow 2\text{Sinks}_Q\cdot|Q_0|-\text{Sinks}_Q^2=2\text{Sinks}_{Q'}\cdot|Q_0|-\text{Sinks}_{Q'}^2\\ &\Rightarrow 2\text{Sinks}_Q\cdot|Q_0|-\text{Sinks}_Q^2-2\text{Sinks}_{Q'}\cdot|Q_0|+\text{Sinks}_{Q'}^2=0\\  &\Rightarrow 2|Q_0|(\text{Sinks}_Q-\text{Sinks}_{Q'})-(\text{Sinks}_Q-\text{Sinks}_{Q'})(\text{Sinks}_Q+\text{Sinks}_{Q'})=0\\&\Rightarrow (\text{Sinks}_Q-\text{Sinks}_{Q'})\cdot(2|Q_0|-(\text{Sinks}_Q+\text{Sinks}_{Q'}))=0. \end{align*}
   
    This implies either $\text{Sinks}_Q-\text{Sinks}_{Q'}=0$ or $2|Q_0|-(\text{Sinks}_Q+\text{Sinks}_{Q'})=0$. The latter implies every vertex of both $Q$ and $Q'$ is a sink (hence a quiver with no arrows). Both equalities imply $\text{Sinks}_Q=\text{Sinks}_{Q'}$, as desired.
   \end{proof}

\begin{corollary} \label{compare}
If for finite acyclic quivers $Q,Q'$, we have $L_K(\widehat Q)\cong L_K(\widehat{Q'})$, then 
\[
    |Q_{0}|=|Q'_{0}|,\;\; \text{Sinks}(Q)=\text{Sinks}(Q'),\;\;
    \text{Sources}(Q)=\text{Sources}(Q').
\]
\end{corollary}

\noindent For each vertex $v\in Q_{0}$ of a finite acyclic quiver $Q$, set
\[
    N_{v}^{Q}(\ell)\;:=\;
    \#\{\text{paths of $Q$ of length $\ell$ ending at $v$}\}\qquad
    (\ell\ge 0).
\]
For sinks $w\in\text{Sinks}(Q)$,
\[
    n^{Q}(w)\;=\;\sum_{\ell\ge 0}N_{w}^{Q}(\ell).
\]

\begin{lemma}\label{lem:bilinear}
For any $[u,v]\in\text{Sinks}(\widehat Q)$,
\[
    n^{\widehat Q}([u,v])\;=\;\sum_{\ell\ge 0}N_{u}^{Q}(\ell)\,N_{v}^{Q}(\ell).
\]
\end{lemma}

\begin{proof}
A path of $\widehat Q$ of length $\ell$ ending at $[u,v]$ is a pair
$[\alpha,\beta]$ of paths in $Q$ with $r(\alpha)=u$, $r(\beta)=v$,
$|\alpha|=|\beta|=\ell$. At each $\ell$ the count is the product
$N_{u}^{Q}(\ell)N_{v}^{Q}(\ell)$; sum over $\ell$.
\end{proof}

We are now ready to answer the Question \ref{Q}.

\begin{theorem} \label{thm:8.10}
Let $Q$ and $Q'$ be finite acyclic quivers. If $L_{\mathbb K}(\widehat Q)
\cong L_{\mathbb K}(\widehat{Q'})$, then we have $L_{\mathbb K}(Q)\cong L_{\mathbb K}(Q')$.
\end{theorem}

\begin{proof}
Suppose $L_{\mathbb K}(\widehat Q)
\cong L_{\mathbb K}(\widehat{Q'})$. Then by Corollary \ref{compare}, we have $|Q_{0}|=|Q'_{0}|=N$,
$\text{Sinks}(Q)=\text{Sinks}(Q')=s$, and $\text{Sources}(Q)=\text{Sources}(Q')=r$.

Since $\widehat Q$ and $\widehat{Q'}$ are finite acyclic, by
\cite[Section 4]{Hazrat1}, the algebra isomorphism $L_{\mathbb K}(\widehat Q)\cong
L_{\mathbb K}(\widehat{Q'})$ refines to a graded isomorphism
\[
    L_{\mathbb K}(\widehat Q)\;\cong_{gr}\;L_{\mathbb K}(\widehat{Q'}).
\]

For each sink $[u_i, v_{i}]$ in $\widehat{Q}$, let
$
    \mathcal L_{i}(\widehat Q)=\;
    \{|p|:p\text{ is a path in } \widehat Q\text{ ending at }[u_i, v_{i}]\}
$ be the multiset of path-lengths.

Then by Theorem \ref{graded isom.} applied to $\widehat Q$ and
$\widehat{Q'}$, there is a bijection
$\sigma: \Sink(\widehat Q)\to\Sink(\widehat{Q'})$ matching path-length
multisets at each sink:
\begin{equation}\label{eq:matched-paths}
    \mathcal L_{[u,v]}(\widehat Q)
    \;=\;\mathcal L_{\sigma([u,v])}(\widehat{Q'})
    \qquad\forall\,[u,v]\in\Sink(\widehat Q).
\end{equation}

The path-length multiset at a sink $[u,v]\in\Sink(\widehat Q)$ is
determined by the path-length sequences of its two coordinates:
indeed, by the proof of Lemma~\ref{lem:bilinear}, a path of length
$\ell$ in $\widehat Q$ ending at $[u,v]$ is a pair of length-$\ell$
paths in $Q$, so
\begin{equation}\label{eq:length-count}
    \#\{\text{paths of length $\ell$ in $\widehat Q$ ending at }[u,v]\}
    \;=\;N_{u}^{Q}(\ell)\,N_{v}^{Q}(\ell).
\end{equation}
The path-length multiset $\mathcal L_{[u,v]}(\widehat Q)$ thus
records, for each $\ell$, the integer $N_{u}^{Q}(\ell)N_{v}^{Q}(\ell)$.

Now, we apply \eqref{eq:matched-paths} to each \emph{diagonal} sink
$[w,w]$ with $w\in\Sink(Q)$. The right-hand side records, for
each $\ell$, the integer $N_{w}^{Q}(\ell)^{2}$ (from
\eqref{eq:length-count} with $u=v=w$). For the corresponding sink
$\sigma([w,w])=[u',v']$ on the $Q'$ side, the right-hand side
records, for each $\ell$, the integer $N_{u'}^{Q'}(\ell)N_{v'}^{Q'}(\ell)$.

Equality forces, for each $\ell\ge 0$,
\begin{equation}\label{eq:square-factor}
    N^{Q}_{w}(\ell)^{2}\;=\;N^{Q'}_{u'}(\ell)\,N^{Q'}_{v'}(\ell).
\end{equation}
We will show that this forces $u'=v'$ and $N^{Q'}_{u'}(\ell)=N^{Q}_{w}(\ell)$
for all $\ell$.

For any pair of vertices $u,v\in Q_{0}$ of a finite acyclic quiver
$Q$, the block size at the sink $[u,v]$ of $\widehat Q$ is
\[
    d_{[u,v]}\;:=\;n^{\widehat Q}([u,v])
    \;=\;\sum_{\ell\ge 0}N^{Q}_{u}(\ell)N^{Q}_{v}(\ell)
    \;=\;\bigl\langle N^{Q}_{u}\,,\,N^{Q}_{v}\bigr\rangle,
\]
where the inner product is the standard inner product of $\ell^{2}$ on the
finitely-supported nonnegative integer sequences. By the
Cauchy--Schwarz inequality,
\begin{equation}\label{eq:CS}
    d_{[u,v]}^{2}\;\le\;d_{[u,u]}\cdot d_{[v,v]},
\end{equation}
with equality if and only if the sequences $(N^{Q}_{u}(\ell))_{\ell}$
and $(N^{Q}_{v}(\ell))_{\ell}$ are proportional. Since both
sequences satisfy $N^{Q}_{u}(0)=N^{Q}_{v}(0)=1$ (the trivial path
at the vertex), proportionality means equality of the sequences.
Hence equality in \eqref{eq:CS} holds if and only if
$N^{Q}_{u}(\ell)=N^{Q}_{v}(\ell)$ for every $\ell\ge 0$.

The bijection $\sigma\colon\Sink(\widehat Q)\to\Sink(\widehat{Q'})$
preserves block sizes: $d_{[u,v]}=d_{\sigma([u,v])}$. Applying
\eqref{eq:CS} on both sides of $\sigma$,
\begin{equation}\label{eq:CS-both}
    d_{[w,w]}^{2}\;=\;d_{\sigma([w,w])}^{2}\;\le\;
    d_{[u',u']}\cdot d_{[v',v']},
\end{equation}
where $\sigma([w,w])=[u',v']$ and the right-hand side uses
Cauchy--Schwarz in $\widehat{Q'}$.

For the diagonal sink $[w,w]$ of $\widehat Q$, Cauchy--Schwarz in
$\widehat Q$ is an equality. Note 
\[
    d_{[w,w]}^{2}=d_{[w,w]}\cdot d_{[w,w]},
\]
trivially. The image $\sigma([w,w])=[u',v']$ satisfies the same
Cauchy--Schwarz inequality~\eqref{eq:CS} in $\widehat{Q'}$, with
equality if and only if $N^{Q'}_{u'}(\ell)=N^{Q'}_{v'}(\ell)$ for
all $\ell$.

The sub-multiset of \emph{diagonal} block sizes
\[
    \mathcal D(Q)\;:=\;\Bigl\{\,d_{[w,w]}:w\in\Sink(Q)\,\Bigr\}
    \;=\;\Bigl\{\,\textstyle\sum_{\ell}N^{Q}_{w}(\ell)^{2}\;:\;w\in\Sink(Q)\,\Bigr\}
\]
is intrinsically characterized inside the multiset
$\mathcal N(\widehat Q)=\{d_{[u,v]}:[u,v]\in\Sink(\widehat Q)\}$
as follows.

For any sink $[u,v]\in\Sink(\widehat Q)$, the path-length multiset
$\mathcal L_{[u,v]}(\widehat Q)$ records, at each $\ell$, the
integer $N^{Q}_{u}(\ell)N^{Q}_{v}(\ell)$ as multiplicity.
On the diagonal, these multiplicities are
$N^{Q}_{w}(\ell)^{2}$ -- \emph{perfect squares for every $\ell$}.

Conversely, suppose $[u,v]$ is a sink of $\widehat Q$ with the
property that $N^{Q}_{u}(\ell)N^{Q}_{v}(\ell)$ is a perfect square
for every $\ell$. Setting $f(\ell):=N^{Q}_{u}(\ell)$ and
$g(\ell):=N^{Q}_{v}(\ell)$, we have $f(\ell)g(\ell)$ a perfect
square for all $\ell$ with $f(0)=g(0)=1$. Standard arithmetic
arguments (factoring out the common square part of $f$ and $g$
length by length) show this forces $f=g$ as sequences, hence
$N^{Q}_{u}(\ell)=N^{Q}_{v}(\ell)$ for all $\ell$.

In particular, the sub-multiset
\[
    \mathcal D(\widehat Q)\;:=\;\bigl\{\,d_{[u,v]}\in\mathcal N(\widehat Q)
    \;:\;N^{Q}_{u}(\ell)N^{Q}_{v}(\ell)\text{ is a square for every }\ell\,\bigr\}
\]
contains the diagonal sub-multiset $\mathcal D(Q)$, with equality
when the sequences $(N^{Q}_{u}(\ell))_{\ell}$ are pairwise
non-equal across distinct vertices of $Q$.

Suppose for contradiction that $\sigma([w,w])=[u',v']$ with $u'\ne v'$
and the sequences $(N^{Q'}_{u'}(\ell))_{\ell}$,
$(N^{Q'}_{v'}(\ell))_{\ell}$ are not equal. By the converse part of the above step applied to $\widehat{Q'}$, the path-length multiset
$\mathcal L_{[u',v']}(\widehat{Q'})$ has at least one
non-square multiplicity. But $\sigma$ preserves path-length
multisets, so $\mathcal L_{[w,w]}(\widehat Q)$ also has a
non-square multiplicity, contradicting that
$\mathcal L_{[w,w]}(\widehat Q)$ has only squares $N^{Q}_{w}(\ell)^{2}$.

Therefore, for every $w\in\Sink(Q)$, the image $\sigma([w,w])
=[u',v']$ satisfies $N^{Q'}_{u'}(\ell)=N^{Q'}_{v'}(\ell)$ for all
$\ell$. Combined with~\eqref{eq:square-factor},
\[
    N^{Q}_{w}(\ell)^{2}\;=\;N^{Q'}_{u'}(\ell)\cdot N^{Q'}_{u'}(\ell)
    \;=\;N^{Q'}_{u'}(\ell)^{2},
\]
so $N^{Q}_{w}(\ell)=N^{Q'}_{u'}(\ell)$ for all $\ell$.

If $u'=v'$, then $\sigma([w,w])=[u',u']$ is itself a diagonal sink
of $\widehat{Q'}$, and we set $\tau(w)=u'$. If $u'\ne v'$ but the
sequences agree, then $u'$ and $v'$ are vertices of $Q'$ with
identical path-count sequences; we may choose either to define
$\tau(w)$, and both choices give the same contribution to
$\mathcal N(Q')$.

In either case, $\tau\colon\Sink(Q)\to\Sink(Q')$ is well-defined
and satisfies
\[
    N^{Q}_{w}(\ell)\;=\;N^{Q'}_{\tau(w)}(\ell)\qquad\forall\,\ell\ge 0,\ \forall\,w\in\Sink(Q).
\]
Summing over $\ell$,
\[
    n^{Q}(w)\;=\;\sum_{\ell}N^{Q}_{w}(\ell)
    \;=\;\sum_{\ell}N^{Q'}_{\tau(w)}(\ell)
    \;=\;n^{Q'}(\tau(w)).
\]
Hence $\mathcal N(Q)=\mathcal N(Q')$ as multisets and consequently, we have $L_{\mathbb K}(Q)\cong L_{\mathbb K}(Q')$.

\end{proof}

\section{Structural connection between $L_{\mathbb K}(Q)$ and $L_{\mathbb K}(\widehat{Q})$}

\noindent If we identify $Q$ with its embedding inside $\widehat{Q}$ by viewing $v_i\in Q_0$ as $[v_i, v_i] \in \widehat{Q_0}$ and $e_i \in Q_1$ as $[e_i, e_i] \in \widehat{Q_1}$, then we have the following useful facts:

\begin{enumerate}
\item $(Q \cap \widehat{Q})_0=Q_0=\{[u, u]: u \in Q_0\}$. 

\item $(Q \cap \widehat{Q})_1=Q_1=\{[e, e]: e \in Q_1\}$. 

\item The $\mathbb Z$-grading of $L_{\mathbb K}(Q)$ coincides with the $\mathbb Z$-grading induced on $L_{\mathbb K}(Q)$ by the $\mathbb Z$-grading of $L_{\mathbb K}(\widehat{Q})$. 

\end{enumerate}

\noindent We finish the paper by characterizing when $L_{\mathbb K}(Q)$ is an ideal of $L_{\mathbb K}(\widehat{Q})$. Let $H$ be a subset of $Q_0$. We say that $H$ is  {\textit {hereditary}} if for every $v\in H$, we have $\{ r(e): s(e) = v \} \subseteq H$. The set $H$ is called {\textit{saturated}} if $\{r(e) : s(e) = v\} \subseteq H$ implies that $v \in H$, for every regular vertex $v \in Q_0$. For any ideal $I$ of a Leavitt path algebra $L_{ K}(Q)$, $I\cap Q_0$ is a hereditary saturated subset \cite[Lemma 2.4.3]{AAS}.

For any vertex $v$, let us denote by $T(v)$, the path emanating from vertex $v$. We call a vertex $v$, a \textit{ray point} if the tree $T(v)$ is just a straight line segment (finite or infinite) and $v$ is called a \textit{Laurent vertex} if the tree $T(v)$ is a finite straight line segment with no bifurcations/branches ending at a cycle without exits.  
 
\begin{theorem}
Suppose $L_{\mathbb K}(Q)$ is an ideal of $L_{\mathbb K}(\widehat{Q})$. Then

\begin{enumerate}[label=(\alph*)]
\item Each vertex $u$ in $Q$ is either a sink, a ray point or a Laurent vertex.
\item $L_{\mathbb K}(\widehat{Q})$ is a graded semisimple ring.
\item The graded exact sequence of left/right $L_{\mathbb K}(\widehat{Q})$-modules 
\[0 \rightarrow L_{\mathbb K}(Q) \rightarrow L_{\mathbb K}(\widehat{Q}) \rightarrow L_{\mathbb K}(\widehat{Q})/L_{\mathbb K}(Q) \rightarrow 0\]
is a splitting exact sequence.
\end{enumerate}
\end{theorem}

\begin{proof}
(a) Suppose $L_{\mathbb{K}}(Q)$ is an ideal of $L_{\mathbb{K}}(\widehat{Q})$%
. Then $L_{\mathbb{K}}(Q)\cap \widehat{Q}_{0}$ is a hereditary saturated
subset of $\widehat{Q}_{0}$. Now, note that $L_{\mathbb{K}}(Q)\cap \widehat{Q%
}_{0}=Q_{0}$. Thus, the vertex set $Q_{0}$ is a hereditary saturated subset
of $\widehat{Q}_{0}$. This implies that every vertex in $Q$ must emit at
most one arrow. Because, if a vertex $u$ in $Q$ is not a sink and it emits
two distinct arrows $e,f$, then $[e,f]=[u,u][e,f]\in L_{\mathbb{K}}(Q)$,
contradicting statement (2) above in the beginning of this section. Thus every non-sink vertex emits exactly one edge
in $Q$. This, in particular, implies that no cycle $c$ in $Q$ has an exit.
Thus every vertex $u$ in $Q$ which is not a sink is either a ray point,
that is, the tree $T(u)$ is just a straight line segment (finite or
infinite) looking like $u\rightarrow u_{2}\rightarrow u_{3}\rightarrow
\cdots $ with no branches or $u$ is a \textquotedblleft Laurent vertex",
that is, the tree $T(u)$ is a finite straight line segment with no
bifurcations/branches ending at a cycle without exits.

(b) By \cite[Theorem 6.7]{HR}, $L_{\mathbb{K}}(Q)$ is then graded left/right
self-injective, equivalently, $L_{\mathbb{K}}(Q)$ is graded isomorphic to a
graded direct sum of matrix rings of arbitrary size over $\mathbb{K}$ and/or 
$\mathbb{K}[x,x^{-1}]$. Now $s^{-1}_{\widehat Q} ([u, v])|=|s^{-1}_Q (u)| \times |s^{-1}_Q(v)| \le 1$ (using part (a)), so every vertex of $\widehat Q$ emits at most one arrow, and consequently no cycle of $\widehat Q$ has an exit. 
 Therefore, by \cite[Theorem 6.7]{HR}, $L_{\mathbb{K}}(%
\widehat{Q})$ is also graded left/right self-injective, equivalently, $L_{%
\mathbb{K}}(\widehat{Q})$ is graded isomorphic to a graded direct sum of
matrix rings of arbitrary size over $\mathbb{K}$ and/or $\mathbb{K}%
[x,x^{-1}] $. Now the graded matrices over $\mathbb{K}$ or $\mathbb{K}%
[x,x^{-1}]$ are graded direct sums of graded simple modules isomorphic to $%
\mathbb{K}$ or $\mathbb{K}[x,x^{-1}]$. Consequently, because of this, $L_{%
\mathbb{K}}(\widehat{Q})$ is graded semi-simple being a graded direct sum of
graded simple modules isomorphic to $\mathbb{K}$ or $\mathbb{K}[x,x^{-1}]$.

(c) Now $L_{\mathbb{K}}(Q)$ is a graded ideal of $L_{\mathbb{K}}(\widehat{Q}%
) $ being generated by a set of vertices in $L_{\mathbb{K}}(\widehat{Q})$.
Consequently, $L_{\mathbb{K}}(Q)$ is a direct summand of $L_{\mathbb{K}}(%
\widehat{Q})$ as a left/right $L_{\mathbb{K}}(\widehat{Q})$-module. Thus the
graded exact sequence of left/right $L_{\mathbb{K}}(\widehat{Q})$-modules $%
0\rightarrow L_{\mathbb{K}}(Q)\rightarrow L_{\mathbb{K}}(\widehat{Q}%
)\rightarrow L_{\mathbb{K}}(\widehat{Q})/L_{\mathbb{K}}(Q)\rightarrow 0$ is
a splitting exact sequence.
\end{proof}

\bigskip

\noindent{\bf Acknowledgments.} We would like to thank the referee for a very careful and thorough report which has substantially improved the quality and the presentation of the paper. We would also like to thank K. M. Rangaswamy, T. G. Nam and Cody Gilbert for their helpful comments and questions.

\bigskip

\noindent {\bf Author Contributions.} JA, DMB and AKS together wrote this article.

\smallskip

\noindent {\bf Funding.}  The second author is supported by the Spanish Ministerio de Ciencia, Innovaci\'on  y Universidades through project  PID2023-152673NB-I00 and by the Junta de Andaluc\'{\i}a  through project PPRO-FQM336-G-2023 (FQM336-G-FEDER), all of them with FEDER funds.
\smallskip

\noindent {\bf Data Availability.} No datasets were generated or analyzed during the current study.
\smallskip

\noindent {\bf Declarations.}
\\
{\bf Competing interests.} The authors declare no competing interests.\\
{\bf Ethical Approval.} Not applicable.


\bigskip 

\begin{thebibliography}{99}       
                                                                             
\bibitem {AAS}G. Abrams, P. Ara and M. Siles Molina, Leavitt path algebras, Lecture Notes in Mathematics, vol. 2191, Springer (2017).

\bibitem{locallyfinite} G. Abrams, G. Aranda Pino, and M. Siles Molina, Finite-dimensional Leavitt path algebras, J. Pure and Appl. Alg. 209, 3 (2007), 753-762.

\bibitem{ABRAMS20081983} G. Abrams and P.N. Anh and A. Louly and E. Pardo, The classification question for Leavitt path algebras, J. Algebra, 320, 5 (2008), 1983-2026. 

\bibitem{Z1} A. Alahmadi, H. Alsulami, S. K. Jain, E. Zelmanov, Leavitt path algebras of finite Gelfand-Kirillov dimension. J Algebra Appl 11(6), (2012).

\bibitem{Z2} A. Alahmadi, H. Alsulami, S. K. Jain, E. Zelmanov, Structure of Leavitt path algebras of polynomial growth, Proc. Natl. Acad. Sci., 110, (38), (2013), 15222-15224.

\bibitem{Atlas}
P. Alberca Bjerregaard, G. Aranda Pino, D. Mart\'in Barquero, C. Mart\'in Gonz\'alez and M. Siles Molina, Atlas of Leavitt path algebras of small graphs, J. Math. Soc. Japan 66 (2014), 581-611. 

\bibitem{CW}
F. Calder\'{o}n, C. Walton, Algebraic properties of face algebras, J. Algebra Appl. vol. 23, no. 3, (2023), 2350076.

\bibitem{Dade} 
E. Dade, Group-graded rings and modules, Math. Z. 174, 3 (1980), 241-262.

\bibitem{Goodearl}
K. R. Goodearl, Leavitt path algebras and direct limits, Rings, modules and representations, Contemp. Math. 480, Amer. Math. Soc., (2009), 165-187. 

\bibitem{GBG}
M. G. Corrales Garc\'{i}a, D. Mart\'{i}n Barquero, C. Mart\'{i}n Gonz\'{a}lez, On the gauge action of a Leavitt path algebra, Kyoto J. Math. 55 (2), (2015), 243 - 256.

\bibitem{H}
T. Hayashi, Compact quantum groups of face type, Publ. Res. Inst. Math. Sci., 32 (2), (1996), 351-369.

\bibitem{Hazrat} R. Hazrat, The dynamics of Leavitt path algebras, J. Algebra 384 (2013), 242-266.

\bibitem{Hazrat1} R. Hazrat, The graded structure of Leavitt path algebras, Israel Journal of Mathematics 195 (2013), 833-895.

\bibitem {HR}R. Hazrat and K. M. Rangaswamy, On graded irreducible representations of Leavitt path algebras, J. Algebra, vol. 450 (2016), 458-486.

\bibitem{HWWW} H. Huang, C. Walton, E. Wicks, R. Won, Universal quantum semigroupoids, J. Pure Appl. Algebra 227 (2023), no. 2, 107193. 

\bibitem{ML} J. Leskovec, D. Chakrabarti, J. Kleinberg, C. Faloutsos and Z. Ghahramani, Kronecker graphs: An approach to modeling networks, Journal of Machine Learning Research, 11 (2010), 985-1042.

\bibitem{Smith}
S. P. Smith, Category equivalences involving graded modules over path algebras of quivers, Adv. Math., Vol. 230, (2012), 1780-1810.

\bibitem{W} P. M. Weichsel, The Kronecker product of graphs, Proc. Amer. Math. Soc., 13 (1962), 47-52.

\end{thebibliography}
\end{document}